\documentclass[a4paper, 12pt, oneside]{amsart}

\usepackage[english]{babel}

\usepackage[a4paper,
            textwidth=155mm, textheight=215mm,
            hcentering,
            top=20.4mm, includehead]{geometry}

\usepackage{amsmath, amsfonts, amssymb, amscd, amsthm}
\usepackage{mathtools}
\usepackage{mathrsfs}
\usepackage[thinc]{esdiff}
\usepackage{xspace}

\usepackage{enumitem}

\usepackage{tikz}
\usetikzlibrary{calc}

\usepackage{xurl}   

\usepackage[hypertexnames=false, linktocpage=true]{hyperref}
\hypersetup{
  pdftitle  = {Fatou Components of Attracting Skew Products: From Polynomial Classification to Transcendental Complexity},
  pdfauthor = {Tom Potthink},
}
\usepackage{cleveref}

\usepackage{thmtools}
\usepackage{thm-restate}

\declaretheorem[name=Lemma,       Refname=Lemma,       parent=section]{lemma}
\declaretheorem[name=Theorem,     Refname=Theorem,     sibling=lemma]{theorem}

\declaretheorem[name=Claim,       Refname=Claim,       sibling=lemma]{claim}

\declaretheorem[style=definition, name=Definition, Refname=Definition, sibling=lemma]{definition}

\declaretheorem[style=remark, name=Remark, Refname=Remark, sibling=lemma]{remark}

\newenvironment{subproof}[1][\proofname]{%
  \begin{proof}[#1]%
}{%
  \end{proof}%
}

\DeclarePairedDelimiter\abs{\lvert}{\rvert}

\DeclarePairedDelimiter\norm{\lVert}{\rVert}
\DeclarePairedDelimiter\paren{(}{)}

\DeclarePairedDelimiter\set{\lbrace}{\rbrace}

\newcommand{\CC}{\mathbb{C}}
\newcommand{\DD}{\mathbb{D}}

\newcommand{\NN}{\mathbb{N}_{0}}

\newcommand{\RR}{\mathbb{R}}

\newcommand{\Chat}[1]{\widehat{\CC^{#1}}}
\newcommand{\PC}[1]{\mathbf{P}_{#1}{\mathbb{C}}}

\newcommand{\cF}{\mathcal{F}}

\newcommand{\F}{\mathcal{F}}
\newcommand{\J}{\mathcal{J}}

\newcommand{\dist}{\operatorname{dist}}
\newcommand{\Hol}{\operatorname{Hol}}
\newcommand{\id}{\operatorname{id}}
\newcommand{\im}{\operatorname{Im}}
\newcommand{\re}{\operatorname{Re}}
\newcommand{\wind}{\operatorname{wind}}
\newcommand{\interior}[1]{\mathring{#1}}

\newcommand{\e}{\varepsilon}

\usepackage{microtype}

\title[Fatou Components of Skew Products]{Fatou Components of attracting Skew Products: \\From Polynomial Classification \\to Transcendental Complexity}

\author{Tom Potthink}
\address{Tom Potthink, Univ. Bordeaux, CNRS, Bordeaux INP, IMB, UMR 5251, F-33400 Talence, France}
\email{tom.potthink@math.u-bordeaux.fr}


\begin{document} 
	
\begin{abstract}
    We study the Fatou components of holomorphic skew products $F \colon \CC^2 \to \CC^2, F(z,w) = (f(z) + r(z,w), g(w))$ near an attracting invariant fiber, introducing a refined classification of components on the fiber: \emph{bulging}, \emph{partial bulging}, and \emph{nowhere bulging}.
    For polynomial skew products, we give a complete characterization of when the basin at infinity is non-bulging, and show that in this case non-bulging implies nowhere bulging, ruling out partial bulging entirely.
    In the transcendental setting, the behavior is richer: we construct partially bulging wandering domains and Baker domains, show that any orbitally unbounded Fatou component can be made nowhere bulging by choosing an appropriate function $r$, and prove that for any nonconstant $r$ there exists an entire function $f$ with an escaping wandering domain that fails to be bulging. 
    Bulging thus depends on the interplay between fiber dynamics and perturbation.
\end{abstract}
\maketitle  

\section{Introduction}
A skew product on $\CC^2$ is a holomorphic map $F \colon \CC^2 \to \CC^2$ of the form
\begin{equation*}
    F(z,w) = (f(z,w), g(w))
\end{equation*}
where $f\colon \CC^2 \to \CC$ and $g \colon \CC \to \CC$.
Note that the dynamics in the second coordinate is purely one-dimensional.

Innocent as they appear, the dynamics of skew products is far from trivial and shows features not present in one-dimensional dynamics.
For example, it was shown in \cite{Astorg2016} that there are polynomial skew products (i.e., skew products where both $f$ and $g$ are polynomial) with wandering domains.
This stands in contrast to Sullivan's No Wandering Domains Theorem \cite{Sullivan1985}, which implies that polynomials in one dimension cannot have wandering domains (see also the discussion in \cite{Dujardin2023}).

Wandering domains illustrate one way two-dimensional skew product dynamics departs from the one-dimensional theory.
This paper is concerned with the behavior near invariant fibers.
An \emph{invariant fiber} is a set $\set{w=c} \subset \CC^2$ where $c \in \CC$ is a fixed point of $g$.
On such a fiber, the skew product restricts to the one-dimensional function $F(z,c) = (f(z,c), c)$.
This raises the natural question of how the dynamics on the fiber and the full two-dimensional system relate to each other.
This setting has been explored by several authors, including Lilov \cite{Lilov2004}, Peters and Vivas \cite{Peters2016}, Peters and Smit \cite{Peters2018}, Peters and Raissy \cite{Peters2019}, Ji \cite{Ji2020,Ji2023}, and Ji and Shen \cite{Ji2026} (an overview can be found in \cite{Raissy2017}); the wandering domains of \cite{Astorg2016} also arise near an invariant fiber $\set{w = c}$, in their case with $g'(c) = 1$. 
Much of the focus has been on polynomial skew products near \emph{attracting fibers}, that is, fibers $\set{w=c}$ where $c$ is an \emph{attracting fixed point} of $g$ (i.e., $g(c)=c$ and $\abs{g'(c)}<1$). 
This paper will also focus on attracting fibers, but will treat transcendental skew products as well.

Since the dynamics is invariant under translation, we shall, without loss of generality, assume throughout that $c = 0$, and we write the skew product $F$ as
\begin{equation}\label{eq:skew-product-form}
    F(z,w) = (f(z) + r(z,w), g(w))
\end{equation}
where, with a slight abuse of notation, $f(z) \coloneqq f(z,0)$ denotes the fiber dynamics, known as the \textit{fiber map}, and $r\colon\CC^2\to\CC$ is a holomorphic function vanishing on the fiber $\set{w=0}$. 
Therefore, $r$ can be written as $r(z,w) = w h(z,w)$ for some holomorphic function $h$.
We refer to $r$ as the \emph{perturbation}, since it measures the deviation from the one-dimensional dynamics on the fiber.

As in one dimension, we define the Fatou set $\F(F)$ of $F$ as the locus of normality of its iterates.
Unlike in one dimension, however, this definition is not choice free: normality requires the target space to be compact, and there are several natural compactifications of $\CC^2$ to choose from, namely, the one-point compactification $\Chat{2}$, projective space $\PC{2}$, and the product $\Chat{} \times \Chat{}$.
In this paper, we work with the one-point compactification.
This choice is immaterial for the results presented here, as will be discussed in \Cref{sec:preliminaries}.
The Julia set $\J(F)$ is defined as the complement of the Fatou set in $\CC^2$.

From a geometric perspective, when investigating the dynamics near an invariant fiber, it is natural to study whether stability of a one-dimensional region is inherited by the full two-dimensional system. 
More precisely, we seek to understand under which conditions a one-dimensional Fatou component $U$ of $f$ on the invariant fiber is \emph{bulging}, meaning that $U\times\set{0}$ is contained in the two-dimensional Fatou set of $F$.
Bulging creates a direct link between the dynamics on the fiber and the full skew product near the fiber, potentially reducing the analysis to the well-understood one-dimensional case. 
For \textit{orbitally bounded} Fatou components (in the sense of \Cref{def:orbitally-unbounded}), bulging over an attracting invariant fiber is completely understood as every such component is bulging. 
This was shown by Lilov \cite{Lilov2004} for attracting basins and Siegel disks, and also for parabolic basins over a \emph{superattracting} fiber, that is, for $g'(0) = 0$.
The parabolic case for $g'(0) \neq 0$ was treated in \cite{Ji2020}. 
Consequently, the main objects of interest for this paper are \emph{orbitally unbounded} Fatou components, that is, those with a subsequence of iterates diverging to infinity (see \Cref{def:orbitally-unbounded}).

We refine the notion of bulging with the following definition.
\begin{definition}\label{def:bulging}
    Let $F$ be a skew product as in \eqref{eq:skew-product-form}, and let $U\subset \CC$ be a Fatou component of $f$.
    We say that $U$ is
    \begin{enumerate}
        \item \emph{bulging} if $U \times \set{0} \subset \F(F)$, and
        \item \emph{non-bulging} if $U \times \set{0} \not\subset \F(F)$.
    \end{enumerate} 
    If $U$ is non-bulging, we call it
    \begin{enumerate}
        \setcounter{enumi}{2}
        \item \emph{nowhere bulging} if $U \times \set{0} \subset \J(F)$, and
        \item \emph{partially bulging} if both $(U\times \set{0}) \cap \J(F) \neq \emptyset$ and $(U \times \set{0}) \cap \F(F) \neq \emptyset$. 
    \end{enumerate}
\end{definition}
Observe that the trichotomy of bulging, partial bulging, and nowhere bulging exhausts all possibilities for a Fatou component.
The case of partial bulging appears to be the most intricate, as for both bulging and nowhere bulging, the set $U\times\set{0}$ lies completely within either the Fatou set or the Julia set of $F$.
The three cases are visualized in Figure~\ref{fig:bulging-cases}.

\begin{figure}[ht]
    \centering
    \begin{tikzpicture}
        \newcommand{\panelbase}[1]{%
            \draw (0,-2.1) -- (0,1.9) node[above] {$z$};
            \node at (0,-2.3) {$\set{w=0}$};
            \draw (-1.2,-1.75) -- (1.2,-1.75) node[right] {$w$};
            \draw[line width=1.2pt] (0,-1.3) -- (0,1.3);
            \draw[line width=1.2pt] (-0.05,-1.3) -- (0.05,-1.3);
            \draw[line width=1.2pt] (-0.05,1.3) -- (0.05,1.3);
            \draw[thin] (-0.05,0.95) -- (-0.9,1.55) node[left] {$U$};
            \node at (0,-2.9) {#1};
        }

        \begin{scope}
            \draw[fill=gray!25, draw=gray!60] (0,0) ellipse (0.9 and 1.3);
            \node at (0.48,0) {$\Omega$};
            \panelbase{(a) bulging}
        \end{scope}

        \begin{scope}[xshift=4.6cm]
            \draw[fill=gray!25, draw=gray!60] (0,0.75) ellipse (0.8 and 0.55);
            \draw[fill=gray!25, draw=gray!60] (0,-0.75) ellipse (0.8 and 0.55);
            \node at (0.45,0.75) {$\Omega_1$};
            \node at (0.45,-0.75) {$\Omega_2$};
            \panelbase{(b) partial bulging}
        \end{scope}

        \begin{scope}[xshift=9.2cm]
            \panelbase{(c) nowhere bulging}
        \end{scope}
    \end{tikzpicture}
    \caption{Schematic of the three cases for a Fatou component $U$ of $f$ on the invariant fiber $\set{w=0}$.
    The thicker segment of the fiber is $U \times \set{0}$, and the shaded regions $\Omega$, $\Omega_1$, and $\Omega_2$ represent Fatou components of $F$.}
    \label{fig:bulging-cases}
\end{figure}
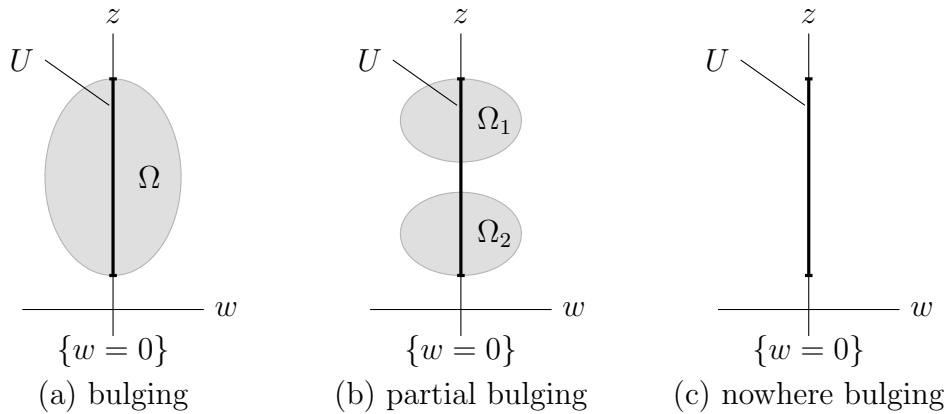

The present paper has three goals.
First, we complete the classification of bulging Fatou components for polynomial skew products by giving a characterization of the bulging of the basin at infinity.
Second, we explore the intricacies that partial bulging and nowhere bulging reveal for transcendental skew products, and give the first examples of partially bulging Fatou components.
Finally, we complement an earlier result from \cite{Potthink2026} to stress the dependence of bulging on the interaction between the fiber dynamics and the perturbation.

In the recent papers \cite{Benini2026} and \cite{Potthink2026}, it was already shown that the connection between the fiber dynamics and the full dynamics becomes significantly more intricate for orbitally unbounded Fatou components.
More precisely, in \cite{Benini2026}, the first example of a non-bulging Fatou component on an attracting fiber was given.
In the later paper~\cite{Potthink2026}, it was shown that for every entire function $f$ with an orbitally unbounded Fatou component $U$ and every nonconstant entire $g$ with a fixed point at $0$, there exists a perturbation $r$ such that $U$ is non-bulging for the resulting skew product given by~\eqref{eq:skew-product-form}. 
Furthermore, in \cite{Potthink2026}, examples of polynomial perturbations $r$ are given for which the basin at infinity of a polynomial $f$ fails to be bulging.
These papers did not address the finer structure of non-bulging, nor whether the polynomial case admits a complete classification.

In this paper, we first complete the analysis for polynomial skew products by characterizing when the basin at infinity is non-bulging in terms of the degree of $f$, the degree of $r$ in $z$ (see \Cref{def:polynomial-in-z}), and the multiplicity of $g$ at its attracting fixed point $0$.
Furthermore, we show that in the case of non-bulging, the basin at infinity is nowhere bulging.
In particular, for such skew products, the basin at infinity is never partially bulging.
Our first result is the following.

\begin{restatable}{theorem}{basin}\label{thm:basin-inf-classification}
    Let $f \colon \CC \to \CC$ be a polynomial of degree $d \geq 2$, let $r \colon \CC^2 \to \CC$ be a polynomial in $z$ of degree $\ell$ vanishing on $\set{w=0}$, and let $g$ be a nonconstant entire function with an attracting fixed point at $0$, where $m \geq 1$ denotes the multiplicity of $g$ at $0$.
    Consider the skew product
    \begin{equation*}
        F(z,w) = (f(z) + r(z,w), g(w)).
    \end{equation*}
    Then the basin at infinity of $f$ is non-bulging for $F$ if and only if $m < d < \ell$.
    Moreover, if the basin is non-bulging, then it is nowhere bulging.
\end{restatable}
The characterizing inequality $m < d < \ell$ reflects the different growth behaviors of $f$ and $r$ and the contraction of $g$ near $0$.
Note that \Cref{thm:basin-inf-classification} does not require $F$ to be a polynomial skew product.
Only the fiber map $f$ needs to be a polynomial and $r$ needs to be polynomial in $z$, while the dependence of $r$ on $w$, as well as the function $g$, may be transcendental.
Combined with Lilov's and Ji's results that orbitally bounded periodic Fatou components on attracting fibers are always bulging and with Sullivan's Theorem on the absence of wandering domains for polynomials, this completes the classification of bulging Fatou components for polynomial skew products, since the basin at infinity is the only orbitally unbounded Fatou component of a polynomial.

The second part of the paper is devoted to contrasting this result with the transcendental setting, which is more nuanced than a simple dichotomy between nowhere bulging and bulging, admitting the intermediate state of partial bulging. 
In \Cref{thm:partial-wandering} and \Cref{thm:partial-baker}, we construct explicit examples of a partially bulging wandering domain and a partially bulging Baker domain.
We also show that any orbitally unbounded Fatou component is nowhere bulging for an appropriate perturbation, strengthening \cite[Theorem 1.1]{Potthink2026} from non-bulging to nowhere bulging.
More precisely, we prove the following.

\begin{restatable}{theorem}{nowhere}\label{thm:nowhere-bulging}
    Let $f$ be an entire function with an orbitally unbounded Fatou component $U \subset \CC$.
    Furthermore, let $g$ be a nonconstant entire function with $g(0) = 0$.
    Then there exists an entire function $h$ such that $U \times \set{0}$ lies in the Julia set of the skew product $F(z,w) = (f(z) + wh(z), g(w))$.
\end{restatable}

Note that \Cref{thm:nowhere-bulging} does not require $0$ to be an attracting fixed point of $g$, but holds for any invariant fiber $\set{w = 0}$.
For attracting fibers, \Cref{thm:nowhere-bulging} yields examples of nowhere bulging Baker domains and wandering domains.
Hence, in the transcendental setting, orbitally unbounded Fatou components can be bulging (see \cite{Benini2026} and \cite{Potthink2026}), partially bulging, and nowhere bulging.

Lastly, in \Cref{thm:tame-perturbation}, we give a statement complementary to \cite[Theorem 1.1]{Potthink2026}, recalled above: whereas there the perturbation $r$ is constructed for given $f$ and $g$, we show, conversely, that given any nonconstant perturbation $r$ and any nonconstant entire $g$ with an attracting fixed point at $0$, there exists an entire function $f$ with an escaping simply connected wandering domain that is non-bulging for the resulting skew product.
In particular, even perturbations that appear well-behaved---for instance, those depending only on $w$---can lead to non-bulging.
The construction relies on results from \cite{Benini2022} and works even when geometric restrictions, in the form of encapsulating Jordan domains, are imposed on the wandering domain (as described in \Cref{thm:tame-perturbation}).
Together with the earlier result from \cite{Potthink2026}, this shows that the bulging of orbitally unbounded Fatou components depends on a nuanced interaction between the fiber dynamics and the perturbation.

\subsection*{Outline of the paper}
In \Cref{sec:preliminaries}, we review several preliminary notions required for the rest of the paper, including results from complex approximation theory.
In \Cref{sec:basin-at-infinity}, we provide the proof of \Cref{thm:basin-inf-classification}.
In \Cref{sec:nowhere-bulging}, we prove the nowhere bulging result for transcendental skew products.
The examples for partial bulging are given in \Cref{sec:partial-bulging}.
Lastly, in \Cref{sec:tame-perturbation}, we construct an entire function $f$ producing a non-bulging wandering domain for any given nonconstant perturbation.

\subsection*{Notation} 
$\NN$ denotes the nonnegative integers.
$\DD$ denotes the open unit disk in $\CC$.
We denote by $A^c$ the complement of a set $A \subset \CC$, and by $\abs{A} \coloneqq \set{\abs{a} \mid a \in A}$ its elementwise absolute value.
Addition and multiplication of sets $A,B \subset \CC$ are understood to be elementwise, i.e., $A + B = \set{a + b \mid a \in A \text{ and } b\in B}$ and $AB = \set{ab \mid a \in A \text{ and } b\in B}$.
We denote by $\pi_z \colon \CC^2 \to \CC, \quad \pi_z(z,w) = z$ and $\pi_w \colon \CC^2 \to \CC, \quad \pi_w(z,w) = w$ the projections onto the $z$- and $w$-coordinates respectively.
Additionally, where there is no risk of confusion, we write $(z_k, w_k) \coloneqq F^{k}(z,w)$ for $k \in \NN$, $(z,w) \in \CC^2$, and $F\colon \CC^2 \to \CC^2$.
Given a function $\phi$ defined on a set $E \subset \CC$, we write $\norm{\phi}_{E} \coloneqq \sup_{z\in E}\abs{\phi(z)} \in [0,\infty]$. 
Since all norms on $\CC^2$ are equivalent, we use $\norm{(z,w)} = \abs{z} + \abs{w}$ for $(z,w) \in \CC^2$.
Given two complex manifolds $X$ and $Y$, we denote by $\Hol(X,Y)$ the space of all holomorphic maps from $X$ to $Y$. 
Lastly, we use $\dist(A,B)$ to denote the distance between two sets $A$ and $B$, where points are treated as singleton sets and curves as their images.

\subsection*{AI statement}
All mathematical results and proofs in this work are due to the author. 
During the preparation of this work the author used Claude Opus 5.5 (Anthropic) to scan the proofs for potential errors, to suggest minor improvements to the language, and to produce a first draft of the TikZ code for the figures. 
All issues raised were independently checked by the author, and any necessary corrections were made by the author, who takes full responsibility for the content of this publication.

\subsection*{Acknowledgments}
I am grateful to my supervisor, Jasmin Raissy, for her guidance and many helpful discussions.

This project was supported in part with funding from the ANR TIGerS/ANR-24-CE40-3604 and the Institut Universitaire de France (IUF).

The author would like to thank the Isaac Newton Institute for Mathematical Sciences, Cambridge, for support and hospitality during the programme Complex dynamics: interactions and influences, where work on this paper was undertaken. This work was supported by EPSRC grant EP/Z000580/1.

\section{Preliminaries}\label{sec:preliminaries}

First, we recall the following definitions and statements from one-dimensional dynamics.

Given an entire function $f$, the \emph{Fatou set} $\F(f)$ of $f$ is defined as the set of normality of the iterates $\lbrace f^k \mid k \in \NN \rbrace$ viewed as functions in $\Hol(\CC, \Chat{})$, where $\Chat{} = \CC \cup \set{\infty}$ is the Riemann sphere.
The connected components of the Fatou set are called \emph{Fatou components} of $f$. 
Given a Fatou component $U \subset \CC$ of $f$, the image $f^n(U)$ is contained in some Fatou component $U_n$ of $f$ for every $n \in \NN$, and we call $U$:
\begin{enumerate}
    \item \emph{periodic} if there exists $n \geq 1$ such that $U = U_n$, and \emph{invariant} if this holds for $n = 1$,
    \item \emph{preperiodic} if $U_n = U_m$ for some $n \neq m$, and
    \item \emph{wandering} if $U_n \neq U_m$ for all $n \neq m$.
\end{enumerate}

The Fatou Classification Theorem (see, for instance, \cite[Theorem 6]{Bergweiler1993} for more details) gives a complete classification for (pre-)periodic Fatou components. 
In particular, if $U$ is an invariant Fatou component of an entire function $f$, then either:
\begin{enumerate}
    \item[(i)] $U$ is an \emph{immediate attracting basin}, meaning that all points in $U$ converge to an \emph{attracting fixed point} $z_0 \in U$, that is, a fixed point of $f$ which satisfies $\abs{f'(z_0)}<1$, or
    \item[(ii)] $U$ is the \emph{basin at $\infty$}, meaning that $f^n \to \infty$ locally uniformly on $U$ as $n\to\infty$, where $f$ is a polynomial, so that $\infty$ is a \emph{superattracting fixed point} of $f$, i.e., $\phi(w) \coloneqq 1/f(1/w)$ satisfies $\phi(0)=0$ and $\phi'(0)=0$, or
    \item[(iii)] $U$ is a \emph{parabolic domain}, that is, all points in $U$ converge to a fixed point $z_0 \in \partial U$ which satisfies $f'(z_0) = 1$, or
    \item[(iv)] $U$ is a \emph{rotation domain}, that is, there exists a subsequence $n_k \xrightarrow{k\to\infty} \infty$ such that $f^{n_k} \to \id$ locally uniformly on $U$ as $k \to \infty$, or
    \item[(v)] $U$ is a \emph{Baker domain}, that is, $f^{n}(z) \to \infty$ for all $z \in U$ as $n \to \infty$ where $f$ is transcendental.         
\end{enumerate}

The No Wandering Domains Theorem due to Sullivan \cite{Sullivan1985} states that for rational functions of degree $d\geq 2$, every Fatou component is preperiodic. Therefore, for rational functions of degree $d\geq 2$, we have a complete description of the dynamics in the Fatou set. 
On the other hand, prior to Sullivan's result, the first example of an entire function with a wandering domain was found by Baker \cite{Baker1976}.
Wandering domains are called \emph{orbitally (or dynamically) bounded}, \emph{oscillating}, or \emph{escaping} according to whether the limit functions of the iterates are all bounded, include both infinity and bounded limits, or consist solely of infinity.
The question of the existence of orbitally bounded wandering domains, an open question in complex dynamics for a long time, has recently been answered in the negative.
It was shown in \cite{Prochorov2026}, \cite{Drach2026}, and \cite{Ye2026} that orbitally bounded wandering domains do not exist.

The notion of orbital boundedness extends to arbitrary Fatou components as follows.

\begin{definition}\label{def:orbitally-unbounded}
    Let $f$ be an entire function and $U \subset \CC$ a Fatou component of $f$.
    We call $U$ \emph{orbitally unbounded} if there exist $z \in U$ and a subsequence $(f^{n_k})_{k\in\NN}$ of iterates such that $f^{n_k}(z) \to \infty$ as $k \to \infty$.
    Otherwise, we call $U$ \emph{orbitally bounded}.
\end{definition}
Note that the condition for being orbitally unbounded is independent of the choice of $z \in U$, since on a Fatou component, every limit function is either in $\Hol(\CC,\CC)$ or identically infinity.
An invariant Fatou component is orbitally unbounded if and only if it is a basin at infinity or a Baker domain, and a wandering domain is always orbitally unbounded.

Also recall the definition of an absorbing domain.

\begin{definition}\label{def:absorbing-domain}
    Let $f$ be an entire function and $U \subset \CC$ a Fatou component of $f$.
    A domain $V \subset U$ is called \emph{absorbing in $U$ for $f$} if $f(V) \subset V$ and if for every $z \in U$, there exists $n \in \NN$ such that $f^n(z) \in V$.
\end{definition}

We now turn to holomorphic maps from $\CC^2$ to $ \CC^2$.
In higher dimensions, the notion of normality, and hence the Fatou set, depends on the choice of compactification. 
For $\CC^2$, the natural compactifications are the one-point compactification $\Chat{2}$, the projective space $\PC{2}$, and the product $\Chat{} \times \Chat{}$ of two Riemann spheres.
In this paper, we work with $\Chat{2}$, which leads to the following definition.

\begin{definition}[{\cite[Definition 1.1]{Wu1967}}]
    Let $X$ be a complex manifold. 
    A family $ \cF \subset \Hol(X, \CC^2) $ of holomorphic maps is called \emph{normal} if every sequence in $\cF$ has a subsequence that either converges uniformly on compact subsets to a limit in $\Hol(X, \CC^2)$ or diverges uniformly to infinity on compact subsets.
\end{definition}

Since divergence to infinity corresponds to convergence to the constant map $\infty$ in $\Chat{2}$, this definition captures normality with respect to the one-point compactification.
We can now define the Fatou set in two dimensions.

\begin{definition}\label{def:fatou-2d}
    Let $F\colon \CC^2 \to \CC^2 $ be a holomorphic map.
    The \emph{Fatou set} $\F(F)$ of $F$ is the set of points in $\CC^2$ that have a neighborhood on which the family of iterates $(F^n)_{n\in\NN}$ is normal. 
    The \emph{Julia set} $\J(F)$ of $F$ is the complement $\CC^2 \setminus \F(F)$. 
\end{definition}

Two observations show that this choice of compactification does not weaken our results, and a third shows that it causes no technical difficulties.
First, by \cite[Lemma 2.4]{Arosio2019}, normality with respect to $\PC{2}$ implies normality with respect to $\Chat{2}$.
The same holds for $\Chat{} \times \Chat{}$, with an analogous proof.
Hence every point we show to lie in $\J(F)$ also lies in the Julia set with respect to $\PC{2}$ and $\Chat{} \times \Chat{}$.
This covers our statements about Julia sets.

Second, if the iterates $(F^n)_{n\in\NN}$ are locally bounded in the $w$-coordinate on an open set $U$, then the three notions of normality agree on $U$.
This is the case, for instance, near an attracting fiber, where the $w$-coordinate converges locally uniformly.
This covers our statements about Fatou sets.

Third, unlike $\PC{2}$ and $\Chat{}\times \Chat{}$, the space $\Chat{2}$ does not carry a complex structure.
Nevertheless, no problems arise with limit functions: on a Fatou component $U\subset \CC^2$, every limit function of $(F^n)_{n\in\NN}$ either lies in $\Hol(U, \CC^2)$ or is identically $\infty$.
The interested reader can find more details in \cite[Section 2]{Potthink2026}.

Next, we recall that the backward invariance of the Fatou set transfers from one- to higher-dimensional dynamics with an analogous proof.

\begin{lemma}\label{lemma:backward-invariance}
    For any holomorphic map $F\colon \CC^2 \to \CC^2$, the Fatou set is backward invariant, that is,
    \begin{equation*}
        F^{-1}(\F(F)) \subset \F(F).
    \end{equation*} 
\end{lemma} 

Furthermore, we will use several results from complex approximation theory, including Runge's Theorem, Arakelyan's Theorem, and Nersesjan's Theorem.
The first result that we shall need is the following version of Runge's Approximation Theorem, which follows directly from \cite[p. 460]{Eremenko1987} (see also \cite[Lemma 5.1]{Benini2022}).

\begin{theorem}[Runge, {\cite[p. 460]{Eremenko1987}}]\label{thm:runge-1}
	Let $(E_n)_{n\in\NN}$ be a sequence of compact subsets of $\CC$ such that
	\begin{enumerate}
		\item the sets $\CC \setminus E_n$ are connected for all $n\in\NN$,
		\item the sets are pairwise disjoint, i.e., $E_n \cap E_m = \emptyset$ for all $n\neq m$, and
		\item $\min\lvert E_n \rvert = \min\set{\lvert z \rvert \mid z \in E_n} \to \infty$ as $n \to \infty$.
	\end{enumerate}
	Suppose that $\psi$ is holomorphic in a neighborhood of $E = \bigcup_{n\in\NN} E_n$, and let $(\e_n)_{n\in\NN}$ be a sequence of positive real numbers. Then there exists an entire function $h$ such that
	\begin{equation*}
		\abs{h(z) - \psi(z)} < \e_n
	\end{equation*}
	for all $z \in E_n$ and $n\in\NN$.
\end{theorem}

We now introduce Weierstraß sets.

\begin{definition}[{\cite[p. 142]{Gaier1987}}]
	Let $G \subset \CC$ be open and connected. A relatively closed subset $E \subset G$ is called a \emph{Weierstraß set} if every continuous function ${f \colon E \to \CC}$ which is holomorphic on the interior $\interior{E}$ can be uniformly approximated on $E$ by functions in $\Hol(G,\CC)$, i.e., for every $\e > 0$ there exists $g \in \Hol(G, \CC)$ such that
	\begin{equation*}
		\abs{f(z) - g(z)} < \e
	\end{equation*} 
	for all $z \in E$.
\end{definition}

We have the following characterization of Weierstraß sets.

\begin{theorem}[Arakelyan, {\cite[p. 142]{Gaier1987}}]\label{thm:arakelyan}
	Let $G \subset \CC$ be open and connected. A relatively closed subset $E \subsetneq G$ is a Weierstraß set if and only if the following properties hold:
	\begin{enumerate}
		\item $\widehat{G}\setminus E$ is connected, and
		\item $\widehat{G}\setminus E$ is locally connected at $\infty$,
	\end{enumerate}
    where $\widehat{G}$ is the one-point compactification of $G$.
\end{theorem}

Next, we introduce Carleman sets.

\begin{definition}[{\cite[p. 145]{Gaier1987}}]
	Let $G \subset \CC$ be open and connected. A relatively closed subset $E \subset G$ is called a \emph{Carleman set} if for every continuous function $f \colon E \to \CC$ which is holomorphic on the interior $\interior{E}$ and for every $\e \colon E \to (0,\infty)$ continuous, there exists $g \in \Hol(G, \CC)$ satisfying
	\begin{equation*}
		\abs{f(z) - g(z)} < \e(z)
	\end{equation*}
	for all $z \in E$.
\end{definition}

The following result gives a characterization of Carleman sets.

\begin{theorem}[Nersesjan, {\cite[p. 157]{Gaier1987}}]\label{thm:nersesjan}
	Let $G \subset \CC$ be open and connected. A relatively closed subset $E \subsetneq G$ is a Carleman set if and only if the following properties hold:
	\begin{enumerate}
		\item $\widehat{G}\setminus E$ is connected,
		\item $\widehat{G}\setminus E$ is locally connected at $\infty$, and
		\item for every compact $K \subset G$ there exists a neighborhood $V \subset \widehat{G}$ of $\infty$ such that no component of the interior $\interior{E}$ intersects both $K$ and $V$,
	\end{enumerate}
    where $\widehat{G}$ is the one-point compactification of $G$.
\end{theorem}

Also recall Weierstraß's Theorem on uniform convergence of holomorphic functions.

\begin{theorem}[Weierstraß, {\cite[p. 176]{Ahlfors1979}}]\label{thm:weierstrass}
    Let $G_1 \subset G_2 \subset \dots$ be an increasing sequence of domains in $\CC$ and $(f_n)_{n\in\NN}$ a sequence of holomorphic functions $f_n\colon G_n \to \CC$. Let $G \coloneqq \bigcup_{n\in\NN} G_n$ and assume that $(f_n)_{n\in\NN}$ converges uniformly on any compact subset of $G$ to a limit function $f$. Then $f$ is holomorphic on $G$.
\end{theorem}

Lastly, we use the following definition for polynomials in one variable.

\begin{definition}\label{def:polynomial-in-z}
    A holomorphic function $r \colon \CC^2 \to \CC$ is called a \emph{polynomial in $z$ of degree $\ell$} if it can be written as
    \begin{equation*}
        r(z,w) = \sum_{j = 0}^{\ell} r_j(w) z^j
    \end{equation*}
    for entire functions $r_0, \dots, r_\ell$ with $r_\ell \not\equiv 0$.
    By convention, we call the constant zero function a polynomial in $z$ of degree $-\infty$.
\end{definition}

\section{Proof of \texorpdfstring{Theorem~\ref{thm:basin-inf-classification}}{the classification theorem}}\label{sec:basin-at-infinity}
This section is devoted to the proof of \Cref{thm:basin-inf-classification}.

\basin* 

The proof of \Cref{thm:basin-inf-classification} is split into three lemmata. We first show the two main implications separately and then prove that non-bulging implies nowhere bulging in this setting.
\begin{lemma}\label{lemma:basin-inf-classification-oneway}
    Under the hypotheses of \Cref{thm:basin-inf-classification}, if $d \geq \ell$ or $m \geq d$, then the basin at infinity of $f$ is bulging.
\end{lemma}

\begin{remark}
    In particular, the basin at infinity is bulging whenever $F$ is a \emph{regular polynomial map of degree $d \geq 2$}, i.e., the components of $F$ are polynomials of degree $d$ and the homogeneous part $F_{\text{hom}}$ of degree $d$ satisfies $F_{\text{hom}}^{-1}(0,0) = \set{(0,0)}$.
    In this case $\ell \leq d$, since the first component of $F$ has degree $d$.
    This can also be deduced more directly from the growth condition
    \begin{equation}\label{eq:regular-poly}
        \liminf_{\norm{(z,w)} \to \infty} \frac{\norm{F(z,w)}}{\norm{(z,w)}^d} > 0,
    \end{equation}
    which is equivalent to regularity (see \cite[p. 158]{Bedford2000}).
    Indeed, from \eqref{eq:regular-poly} it follows that there exists $R > 0$ large enough that for all $(z,w) \in \CC^2$ with $\norm{(z,w)} \geq R$,
    \begin{equation*}
        \norm{F(z,w)} \geq 2 \norm{(z,w)}.
    \end{equation*}
    Consequently, all points in $\set{(z,w) \in \CC^2 \mid \norm{(z,w)} \geq R}$ diverge uniformly to infinity.
\end{remark}

\begin{proof}
    We assume throughout this proof that $\ell \geq 0$.
    Otherwise, $\ell = -\infty$, which implies $r \equiv 0$.
    Then $F$ is a product map and all Fatou components are bulging.

    Since $r$ vanishes on the fiber $\set{w = 0}$, there exists a holomorphic function $h \colon \CC^2 \to \CC$ such that $r(z,w) = wh(z,w)$ for all $(z,w) \in \CC^2$.
    The functions $f$ and $h$ can be written as
    \begin{equation*}
        f(z) = \sum_{j=0}^d a_j z^j \text{ and } h(z,w) = \sum_{j = 0}^{\ell} h_j(w) z^j
    \end{equation*}
    with coefficients $a_0, \dots, a_d \in \CC$ where $a_d \neq 0$, and entire functions $h_0, \dots, h_\ell$ where $h_\ell \not\equiv 0$.

    Up to rescaling in $z$, we may assume that $a_d = 1$.
    Up to a local one-dimensional holomorphic conjugacy in $w$, we may assume that near $w = 0$, that is, on a set $D(0,\delta_g)$ for $0 < \delta_g < 1$ sufficiently small, the function $g$ is of the form $g(w) = \lambda w$ for some $\lambda \in \DD \setminus\set{0}$ if $m =1$ or of the form $g(w) = w^m$ if $m \geq 2$.
    From now on, we work on $\CC \times D(0,\delta_g)$ with $g$ in this normal form.

    For $z \in \CC\setminus\set{0}$ and $w \in D(0,\delta_g)$, the first coordinate function of $F$ can be written as
    \begin{equation}\label{eq:basin-at-infinity-1}
        f(z) + wh(z,w) = z^d \paren*{ 1 + \sum_{j=0}^{d-1} a_j z^{j-d} + w \sum_{k = 0}^{\ell} h_k(w) z^{k - d} }.
    \end{equation}
    Define $A\coloneqq \sum_{j=0}^{d-1} \abs{a_j}$ and $B \coloneqq \sum_{j=0}^{\ell} \sup_{\abs{w} \leq \frac{\delta_g}{2}} \abs{h_j(w)} > 0$. 

    We start with the case $\deg(f) = d \geq \ell$.
    Since $d \geq \ell$, if we fix any $\rho \geq \max\set{4A, 4}$ and $0 < \delta < \min\set*{\frac{1}{4B}, \frac{\delta_g}{2}}$, then for all $z \in \overline{D(0,\rho)}^c$ and $w \in D(0,\delta)$,
    \begin{equation}\label{eq:basin-at-infinity-2}
        \abs*{\sum_{j=0}^{d-1} a_j z^{j-d} + w \sum_{k = 0}^{\ell} h_k(w) z^{k - d}} \leq A \frac{1}{\abs{z}} + B \abs{w} < A \frac{1}{\rho} + B \delta < \frac{1}{2}.
    \end{equation}
    Then, from \eqref{eq:basin-at-infinity-1} and \eqref{eq:basin-at-infinity-2} together with the reverse triangle inequality, it follows that
    \begin{equation*}
        \abs{f(z) + wh(z,w)} \geq \abs{z}^d \paren*{1 - \frac{1}{2}} \geq \abs{z} \frac{\rho}{2} \geq \frac{4}{2} \abs{z} = 2 \abs{z}
    \end{equation*}
    for all $z \in \overline{D(0,\rho)}^c$ and $w \in D(0,\delta)$.
    Inductively, we obtain for the $n$th iterate $(z_n,w_n) = F^n(z,w)$ that
    \begin{equation*}
        \abs{z_n} \geq 2^n \abs{z} \geq 2^n \rho \to \infty
    \end{equation*}
    uniformly as $n \to \infty$ on $\overline{D(0,\rho)}^c \times D(0,\delta)$. 
    This implies that $\overline{D(0,\rho)}^c \times \set{0} \subset \F(F)$.
    Since $\overline{D(0,\rho)}^c$ is an absorbing domain for the basin at infinity, in the sense of \Cref{def:absorbing-domain}, the claim follows using the backward invariance of the Fatou set as established in \Cref{lemma:backward-invariance}.

    We now turn to the second case and assume that $m \geq d$. 
    In particular, $m \geq 2$, and thus, $g(w) = w^m$, as well as $w_n = g^{n}(w) = w^{m^n}$ for all $w \in D(0,\delta_g)$ and $n \in\NN$. 
    Fix any $\rho > \max\set{4A, 2^{d+1}}$ and deduce, as before, that
    \begin{equation}\label{eq:basin-at-infinity-3}
        \abs*{\sum_{j=0}^{d-1} a_j z^{j-d}} \leq A \frac{1}{\rho} \leq \frac{1}{4}
    \end{equation}
    for all $z \in \overline{D(0,\rho)}^c$. Fix any $R > \rho$ and consider the annulus $A(0;\rho,R) \coloneqq \set{ z \in \CC \mid \rho < \abs{z} < R}$.
    Observe that for $w \in D\paren*{0,\frac{\delta_g}{2}}$ and $z \in \overline{D(0,\rho)}^c$,
    \begin{equation}\label{eq:basin-at-infinity-4}
        \abs*{w \sum_{k = 0}^{\ell} h_k(w) z^{k - d}} \leq B \abs{w} \abs{z}^\ell.
    \end{equation}
    Let $ 0 < \delta = \delta(R) < \frac{\delta_g}{2} $ be small enough that $\delta\paren*{\paren*{\frac{3}{2}}^{d} R}^{\ell}< \min\set*{\frac{1}{4B}, 1}$.
    Then, for all $n \in \NN$ and all $(z,w) \in A(0;\rho,R) \times D(0,\delta)$, it follows that
    \begin{equation}\label{eq:basin-at-infinity-5}
        B\paren*{ \abs{w} \paren*{\paren*{\frac{3}{2}}^{d} \abs{z}}^{\ell}}^{m^n} < B \paren*{\min\set*{\frac{1}{4B}, 1}}^{m^n} \leq \frac{1}{4}.
    \end{equation}
    We prove by induction that for all $z \in A(0;\rho,R)$, $ w \in D(0,\delta)$, and $n \in \NN$, we have
    \begin{equation}\label{eq:basin-at-infinity-6}
        \paren*{\frac{1}{2}}^{\frac{d^{n+1} - 1}{d - 1}} \abs{z}^{d^n} \leq \abs{z_n} \leq \paren*{\frac{3}{2}}^{\frac{d^{n+1} - 1}{d - 1}} \abs{z}^{d^n}.
    \end{equation}
    Indeed, fix any $z \in A(0;\rho,R)$ and $w \in D(0,\delta)$.
    The case $n = 0$ is evident. 
    Assume that \eqref{eq:basin-at-infinity-6} holds for some $n \in \NN$. 
    Then, the $n$th iterate $z_n$ of $(z,w)$ satisfies
    \begin{equation*}
        \begin{split}
            \abs{z_n} \leq \paren*{\frac{3}{2}}^{\frac{d^{n+1} - 1}{d - 1}} \abs{z}^{d^n} &\leq \paren*{\frac{3}{2}}^{d^{n+1}} \abs{z}^{d^n} = \paren*{\paren*{\frac{3}{2}}^{d} \abs{z}}^{d^n}.
        \end{split}
    \end{equation*}
    Also observe that
    \begin{equation*}
        \abs{z_n} \geq \paren*{\frac{1}{2}}^{\frac{d^{n+1} - 1}{d - 1}} \abs{z}^{d^n} > \paren*{\frac{1}{2}}^{\frac{d^{n+1} - 1}{d - 1}} \rho^{d^n} \geq \rho
    \end{equation*}
    for $n \geq 1$, while $\abs{z_0} = \abs{z} > \rho$ trivially.
    From \eqref{eq:basin-at-infinity-4}, \eqref{eq:basin-at-infinity-5}, and $d \leq m$, it follows that
    \begin{alignat*}{1}
        \abs*{w_n \sum_{j = 0}^{\ell} h_j(w_n) z_n^{j - d}} &\leq B \abs{w_n} \abs{z_n}^\ell \\
        &\leq B \abs{w}^{m^n}\paren*{\paren*{\frac{3}{2}}^{d} \abs{z}}^{d^n \ell} \\
        &\leq B\paren*{ \delta \paren*{\paren*{\frac{3}{2}}^{d} R}^{\ell}}^{m^n} < \frac{1}{4}.
    \end{alignat*}
    We can apply \eqref{eq:basin-at-infinity-3}, and as in the first case, it follows that
    \begin{equation*}
        \paren*{\frac{1}{2}}^{\frac{d^{n+2} - 1}{d - 1}} \abs{z}^{d^{n+1}} \leq \frac{1}{2} \abs{z_n}^{d} \leq \abs{z_{n+1}} \leq \frac{3}{2} \abs{z_n}^d \leq \paren*{\frac{3}{2}}^{\frac{d^{n+2} - 1}{d - 1}} \abs{z}^{d^{n+1}},
    \end{equation*}
    because $z_{n+1} = f(z_n) + w_nh(z_n,w_n)$.
    This completes the induction.  

    To finish the proof, observe that \eqref{eq:basin-at-infinity-6} implies locally uniform divergence to infinity in the $z$-coordinate on the neighborhood $A(0;\rho,R) \times D(0, \delta)$ of $A(0;\rho,R) \times \set{0}$, since, having chosen $\rho > 2^{d+1}$, we obtain
    \begin{equation*}
        \paren*{\frac{1}{2}}^{\frac{d^{n+1} - 1}{d - 1}} \abs{z}^{d^n} \geq \paren*{\frac{1}{2}}^{d^{n+1}} \rho^{d^n} = \paren*{\frac{1}{2^d}\rho}^{d^n} > 2^{d^n} \to \infty
    \end{equation*}
    as $n \to \infty$.
    Furthermore, since $R > \rho$ was arbitrary, the absorbing domain $\overline{D(0,\rho)}^c$ of the basin at infinity satisfies $\overline{D(0,\rho)}^c \times \set{0} \subset \F(F)$, and hence, as in the first case, the basin at infinity is bulging.
\end{proof}

We now show the other implication.

\begin{lemma}\label{lemma:basin-at-inf-2}
    Under the hypotheses of \Cref{thm:basin-inf-classification}, the basin at infinity is non-bulging if $ m < d < \ell $.
\end{lemma}

\begin{proof}
    We assume that $1 \leq m < d < \ell$.

    There exist $0 < \delta_g < 1$ and a one-dimensional holomorphic conjugacy on $D(0,\delta_g)$ conjugating $g$ to $w \mapsto \lambda w$ for some $\lambda \in \DD \setminus\set{0}$ if $m =1$ or to $ w\mapsto w^m$ if $m \geq 2$.
    From now on, we work on $\CC \times D(0,\delta_g)$ with $g$ in this normal form.

    To prove that the basin at infinity of $f$ is non-bulging, we show that there exists $z_0$ in the basin at infinity of $f$ such that points arbitrarily close to $(z_0, 0)$ have orbits coming back close to the origin $(0,0)$.
    Since $0$ is an attracting fixed point of $g$, this is trivial for the $w$-coordinate.
    For the $z$-coordinate, we show that there exists $R > 0$ such that for every $z_0 \in D(0,R)^c$, every $\e > 0$, and every $M \in \NN $, there exist $\widetilde{w} = \widetilde{w}(\e,M) \in D(0,\e)$ and $n = n(\e,M) \geq M$ with
    \begin{equation}\label{eq:basin-at-infinity-7}
        \pi_z F^{n}(z_0,\widetilde{w}) = 0.
    \end{equation}
    This is sufficient to show non-bulging since it prevents uniform divergence to infinity on any neighborhood of $(z_0,0)$.

    To obtain these points $\widetilde{w}(\e,M)$, the key idea is to compare the growth of $f$, $r$, and $g$.
    This is also where the hypothesis $1 \leq m < d < \ell$ is used.
    This comparison is summarized in the following lemma, which we prove after finishing the proof of \Cref{lemma:basin-at-inf-2}.

    \begin{lemma}\label{lemma:sublemma-classification}
        There exists $R > 0$ with the following property. Let $z_0 \in D(0,R)^c$, $\e > 0$, and $M \in \NN$. Then there exist 
        \begin{itemize}
            \item a radius $0 < \delta < \min\set{\e,\delta_g}$ and an integer $n \geq M$,
            \item integers $\hat{\nu}, \hat{\mu} \geq 1$,
            \item a constant $0 < \alpha < 1$,
            \item holomorphic functions $\hat{r} \colon \CC \times D(0,\delta_g) \to \CC$ and $\phi \colon (\CC\setminus\set{0})\times (D(0,\delta')\setminus\set{0}) \to \CC$ for some $\delta' > \delta$, 
            \item  a holomorphic function $b_{\hat{\mu}} \colon D(0,\delta_g) \to \CC$ satisfying $b_{\hat{\mu}}(0) \neq 0$ and $b_{\hat{\mu}}(w) \in D\paren*{b_{\hat{\mu}}(0), \frac{\abs{b_{\hat{\mu}}(0)}}{2}}$ for all $w \in \overline{D(0,\delta)}$,
        \end{itemize}
        such that for all $w \in \partial D(0,\delta)$, writing $(z_n,w_n) = F^n(z_0,w)$, the following hold:
        \begin{enumerate}[label=(\roman*), itemsep=2ex, topsep=1ex]
            \item \label{prop:r-dominating}
            $\displaystyle \abs*{f(z_n) + r(z_n,w_n) - z_n^d\hat{r}(z_n,w_n)} \leq \alpha \abs{z_n^d\hat{r}(z_n,w_n)};$
            \item \label{prop:r-factorization}
            $\displaystyle \hat{r}(z_n,w_n) = w_n^{\hat{\nu}} z_n^{\hat{\mu}} b_{\hat{\mu}}(w_n) ( 1 + \phi(z_n,w_n) );$
            \item \label{prop:phi-small}
            $\displaystyle \abs{\phi(z_n,w_n)} \leq \tfrac{1}{2};$
            \item \label{prop:zn-nonzero}
            $\displaystyle z_n \neq 0.$
        \end{enumerate}
    \end{lemma}
    We now prove the existence of the $\widetilde{w} \in D(0,\e)$ asserted in \eqref{eq:basin-at-infinity-7}.
    Let $R > 0$ be as in \Cref{lemma:sublemma-classification}.
    Fix any $z_0 \in D(0,R)^c$, $\e > 0$, and $M \in \NN$, and let $\delta$, $n$, $\hat{\nu}$, $\hat{\mu}$, $\alpha$, $\hat{r}$, $\phi$, and $b_{\hat{\mu}}$ be as provided by \Cref{lemma:sublemma-classification} for these choices.
    For $k \in \NN$, consider the holomorphic map
    \begin{equation*}
        h_k \colon D(0,\delta_g) \to \CC, \qquad h_k(w) \coloneqq \pi_z F^k(z_0, w),
    \end{equation*}
    so that $F^k(z_0,w) = (h_k(w), g^k(w))$,
    and define the closed curves $z_k$ and $w_k$ via 
    \begin{equation*}
        (z_k(t),w_k(t)) \coloneqq (h_k(\delta e^{2\pi i t}), g^k(\delta e^{2\pi i t})) = F^k(z_0, \delta e^{2\pi i t})
    \end{equation*} 
    for $t \in [0,1]$.
    Since $D(0,\delta) \subset D(0,\e)$ and $n+1 > M$, in order to prove \eqref{eq:basin-at-infinity-7}, it suffices to show that $h_{n+1}$ has a zero in $D(0,\delta)$.
    By the argument principle, if the curve $z_{n+1}$ does not pass through $0$, then its winding number $\wind(z_{n+1}, 0)$ around $0$ is equal to the number of zeros of $h_{n+1}$ in $D(0,\delta)$, counted with multiplicity.
    Hence, it is enough to show that the curve $z_{n+1}$ does not pass through $0$ and that its winding number around $0$ satisfies $\wind(z_{n+1}, 0) > 0$.

    Since $(z_n(t), w_n(t)) = F^n(z_0, \delta e^{2\pi i t})$ with $\delta e^{2\pi i t} \in \partial D(0,\delta)$, properties~\ref{prop:r-dominating}--\ref{prop:zn-nonzero} of \Cref{lemma:sublemma-classification} hold at $(z_n(t), w_n(t))$ for all $t \in [0,1]$.
    In particular, $z_n$ does not pass through $0$ by property~\ref{prop:zn-nonzero}.
    Since $h_n$ is holomorphic on $D(0,\delta_g) \supset \overline{D(0,\delta)}$, the argument principle shows that $\wind(z_n, 0)$ equals the number of zeros of $h_n$ in $D(0,\delta)$, counted with multiplicity. 
    In particular, we have that
    \begin{equation}\label{eq:nonneg-winding}
        \wind(z_n, 0) \geq 0.
    \end{equation}
    Moreover, since $g$ is in normal form, $w_n(t)$ equals $\lambda^n \delta e^{2\pi i t}$ if $m = 1$ and $\delta^{m^n} e^{2\pi i m^n t}$ if $m \geq 2$.
    Hence, $w_n$ takes values in $\overline{D(0,\delta)} \setminus \set{0}$ and $\wind(w_n, 0) = m^n \neq 0$.

    By definition of $F$, we can write $z_{n+1} = f(z_n) + r(z_n, w_n) = \gamma_1 + \gamma_2$ for the closed curves
    \begin{equation*}
        \gamma_1(t) \coloneqq z_n(t)^d \hat{r}(z_n(t),w_n(t))
    \end{equation*}
    and
    \begin{equation*}
        \gamma_2(t) \coloneqq f(z_n(t)) + r(z_n(t), w_n(t)) - z_n(t)^d \hat{r}(z_n(t),w_n(t)).
    \end{equation*}
    We will reduce the winding number of $z_{n+1}$ to that of $\gamma_1$.
    First, observe that by property~\ref{prop:r-factorization}, $\gamma_1$ factorizes as
    \begin{equation*}
        \gamma_1(t) = z_n(t)^{d+\hat{\mu}} w_n(t)^{\hat{\nu}} b_{\hat{\mu}}(w_n(t))  (1+\phi(z_n(t),w_n(t)))
    \end{equation*}
    for all $t \in [0,1]$.
    None of the four factors passes through $0$. 
    For $z_n$, this follows by property~\ref{prop:zn-nonzero}, and for $w_n$, this was shown above. 
    Next, $b_{\hat{\mu}}(w_n)$ takes values in the disk $D\paren*{b_{\hat{\mu}}(0), \abs{b_{\hat{\mu}}(0)}/2}$, which does not contain $0$ since $b_{\hat{\mu}}(0) \neq 0$. 
    Similarly, $1 + \phi(z_n,w_n)$ only takes values in $\overline{D(1,1/2)}$ by property~\ref{prop:phi-small}.
    In particular, since each of the last two factors stays in a disk away from $0$, their winding numbers around $0$ vanish.
    As winding numbers around $0$ add under multiplication, we obtain using \eqref{eq:nonneg-winding} that
    \begin{equation*}
        \wind(\gamma_1, 0) = (d+\hat{\mu}) \wind(z_n,0) + \hat{\nu} \wind(w_n,0) \geq \hat{\nu} m^n > 0.
    \end{equation*}
    Finally, property~\ref{prop:r-dominating} yields that $\abs{\gamma_2(t)} \leq \alpha \abs{\gamma_1(t)} < \abs{\gamma_1(t)}$ for all $t \in [0,1]$, where we used that $0 < \alpha < 1$ and that $\gamma_1$ does not pass through $0$ as established above.
    Hence, the map $(t,u) \mapsto \gamma_1(t) + u\gamma_2(t)$, where $t,u\in[0,1]$, is a homotopy in $\CC\setminus\set{0}$ between $\gamma_1$ and $z_{n+1} = \gamma_1 + \gamma_2$.
    In particular, $z_{n+1}$ does not pass through $0$, and since the winding number is invariant under homotopy,
    \begin{equation*}
        \wind(z_{n+1}, 0) = \wind(\gamma_1, 0) > 0.
    \end{equation*}
    By the argument principle, $h_{n+1}$ has a zero $\widetilde w \in D(0,\delta)$, that is, $\pi_z F^{n+1}(z_0, \widetilde w) = 0$, which finishes the proof.
\end{proof}

We now give the proof of \Cref{lemma:sublemma-classification}.

\begin{proof}[Proof of {Lemma~\ref{lemma:sublemma-classification}}]
    The proof is elementary but computationally involved. 
    The main task is to find suitable bounds and normalizations for the terms of $f$, $g$, and $r$, which allow us to reduce the analysis to the terms that dominate the growth.

    We first consider the case where $g$ has a superattracting fixed point at $0$, that is, $m \geq 2$.
    The attracting case, $m=1$, is very similar.
    We sketch it at the end of the proof.

    Write $f(z) = \sum_{j=0}^{d} a_j z^j$ for $a_0, \dots, a_d \in \CC$, $a_d \neq 0$. 
    Up to rescaling in $z$, we may assume that $a_d = 1$.
    The function $r$ can be written as $ r(z,w) = \sum_{j = 0}^{\ell} r_j(w) z^j$ where each $r_j$ is holomorphic with $r_j(0) = 0$ and $r_\ell \not\equiv 0$.

    Define
    \begin{equation*}
        \hat{r}(z,w) \coloneqq \frac{r(z,w)- \sum_{j = 0}^{d} r_j(w) z^j}{z^d} = \sum_{j=1}^{\ell-d} r_{d+j}(w) z^{j}.
    \end{equation*}
    Let $\nu_{d+j} \geq 1$ denote the order of the zero of the function $r_{d+j}$ at $0$.
    By convention, we set $\nu_{d+j} = \infty$ if $r_{d+j} \equiv 0$.
    Thus, if $\nu_{d+j} < \infty$, then $r_{d+j}(w) = w^{\nu_{d+j}} b_{j}(w)$ for some holomorphic function $b_{j}\colon D(0,\delta_g) \to \CC$ with $b_{j}(0) \neq 0$.
    Set
    \begin{equation*}
        \mathcal{R} \coloneqq \set{ (\nu_{d+j}, j) \mid j \in\set{1, \dots, \ell-d} \text{ with } \nu_{d+j} \neq \infty}.
    \end{equation*}
    This set is nonempty since $r_\ell \not\equiv 0$ implies $\nu_{\ell} \neq \infty $.
    Furthermore, observe that $\mathcal{R}$ has at most $\ell-d$ elements.
    The function $\hat{r}$ can then be written as
    \begin{equation*}
        \hat{r}(z,w) = \sum_{(\nu,\mu) \in \mathcal{R}} b_{\mu}(w) w^{\nu}z^{\mu}.
    \end{equation*}

    Later, we will need the term that contributes the most to the perturbation, which is the term with the largest ratio of the exponent of $z$ to that of $w$. 
    To that end, let $(\hat{\nu}, \hat{\mu}) \in \mathcal{R}$ be such that
    \begin{enumerate}
        \item $\displaystyle \frac{\hat{\mu}}{\hat{\nu}} \geq \frac{\mu}{\nu}$ for all $(\nu,\mu) \in \mathcal{R}$, and 
        \item $\displaystyle \hat{\mu} \geq \mu$ for all $(\nu,\mu) \in \mathcal{R}$ with $\displaystyle \frac{\hat{\mu}}{\hat{\nu}} = \frac{\mu}{\nu}$.
    \end{enumerate}
    Set $\hat{q} \coloneqq \frac{\hat{\mu}}{\hat{\nu}}$.
    We define $\mathcal{R}_{1} \coloneqq \set{(\nu,\mu) \in \mathcal{R} \mid \frac{\hat{\mu}}{\hat{\nu}} = \frac{\mu}{\nu} \text{ and } \mu < \hat{\mu}}$ and $\mathcal{R}_{2} \coloneqq \set{(\nu,\mu) \in \mathcal{R} \mid \frac{\hat{\mu}}{\hat{\nu}} > \frac{\mu}{\nu}} $.
    Then, $\mathcal{R}$ can be written as the disjoint union $\mathcal{R} = \set*{(\hat{\nu}, \hat{\mu})} \sqcup \mathcal{R}_1 \sqcup \mathcal{R}_2$.
    If $\mathcal{R}_1$ or $\mathcal{R}_2$ is empty, the corresponding steps in the proof can be skipped.

    By continuity of the functions $b_{\mu}$, there exists $\delta_b > 0$ such that for all $w \in D(0,\delta_b)$ and $(\nu,\mu) \in \mathcal{R}$,
    \begin{equation}\label{eq:b-nearconstant}
        \abs{b_{\mu}(w) - b_{\mu}(0) } < \frac{\abs{b_{\mu}(0)}}{2}.
    \end{equation}
    
    Next, we set $B_{+} \coloneqq \frac{3}{2}  \max_{(\nu,\mu) \in \mathcal{R}}  \abs{b_{\mu}(0)} > 0 $ and $B_{-} \coloneqq \frac{1}{2} \min_{(\nu,\mu) \in \mathcal{R}} \abs{b_{\mu}(0)} > 0$.
    Observe that $\frac{2(\ell-d)B_{+}}{B_{-}} \geq 1$.
    Replacing $\frac{3}{2}$ by a larger constant if necessary, we may assume that $B_{+} > 1$.
    Let $0 < C < 1$ and observe that $\frac{C}{2B_{+}(\ell-d)} < 1$.
    Since $r_j(0) = 0$ for all $j \in \set{0,\dots, d}$, continuity yields the existence of $\delta_r > 0$ such that
    \begin{equation}\label{eq:basin-at-infinity-8}
        \sum_{j=0}^{d} \max_{\abs{w} \leq \delta_r} \abs{r_j(w)} \leq \frac{C}{4}.
    \end{equation}

    Set $A \coloneqq \sum_{j=0}^{d-1} \abs{a_j}$ and fix any $R_1 > \max\set*{\frac{4A}{C}, 1}$. 
    Then,
    \begin{equation}\label{eq:basin-at-infinity-10}
        \sum_{j=0}^{d-1} \abs{a_j} R_1^{j-d} \leq A \frac{1}{R_1} \leq \frac{C}{4}.
    \end{equation}

    Set $C_f \coloneqq A + 1$ so that $\abs{f(z)} \leq C_f \abs{z}^{d}$ for all $\abs{z} \geq 1$.
    We define 
    \begin{equation}\label{eq:basin-at-infinity-def-q}
        q \coloneqq \max\paren*{\set*{\frac{\mu}{\nu} \,\middle|\, (\nu,\mu) \in \mathcal{R}_2} \cup \set{0}}.
    \end{equation}
    Observe that $q < \hat{q} = \frac{\hat{\mu}}{\hat{\nu}}$.
    Furthermore, set $\sigma \coloneqq \max\set*{q, \frac{m}{d} \hat{q}} < \hat{q}$.

    Let $0 < \alpha < 1$.
    Since $d > m$, there exist constants $0 < C_1 < 1$ and $C_2 > 1$ such that for all $n \in \NN$,
    \begin{align}
    C_1^{d^n} &\leq \paren*{\frac{C}{2B_{+}(\ell-d)}}^{m^n}
                    (1+C)^{-\frac{d^{n-1} - 1}{d - 1} m\hat{q}},
        \label{eq:basin-at-infinity-C1-a}\\
    C_1^{d^n} &\leq \frac{B_{-}}{2(\ell-d)B_{+}}\,
                    (1+C)^{-\frac{d^n - 1}{d - 1} q},
        \label{eq:basin-at-infinity-C1-b}\\
    C_2^{d^n} &\geq \paren*{\frac{2(C_f + 1)}{\alpha B_{-}}}^{\frac{1}{\hat{\nu}}}
                    \paren*{\frac{1+C}{1-C}}^{d\frac{d^n-1}{\hat{\nu}(d-1)}}
                    (1-C)^{-\frac{d^n - 1}{d - 1} \hat{q}},
        \quad\text{and} \label{eq:basin-at-infinity-C2-a}\\
    C_2^{d^n} &\geq \frac{2(\ell-d)B_{+}}{B_{-}}\,
                    (1-C)^{-\frac{d^n - 1}{d - 1} \hat{q}}.
        \label{eq:basin-at-infinity-C2-b}
    \end{align}
    Indeed, every factor on the right-hand sides is either a constant or of the form $\kappa^{m^n}$ or $\kappa^{\frac{d^n-1}{d-1}}$ for some $\kappa > 0$. 
    Since $d > m$, each such factor lies between $\kappa_1^{d^n}$ and $\kappa_2^{d^n}$ for suitable constants $0 < \kappa_1 \leq 1 \leq \kappa_2$.

    There exists $R_2 > 0$ large enough so that $\frac{C_1}{C_2} R_2^{\hat{q} - \sigma} > 1$ and $\frac{C_1}{R_2^{\sigma}} < 1$. 
    Hence, for all $z \in D(0,R_2)^c$,
    \begin{equation}\label{eq:basin-at-infinity-11}
        \paren*{\frac{C_1}{C_2}\abs{z}^{\hat{q} - \sigma}}^{d^n} \to \infty
    \end{equation}
    as $n \to\infty$, and, since $d > m$, for any $\e' > 0$,
    \begin{equation}\label{eq:basin-at-infinity-12}
        \frac{\paren*{\frac{C_1}{\abs{z}^{\sigma}}}^{d^n}}{\paren*{\frac{\e'}{2}}^{m^n}} = \paren*{ \frac{C_1}{\abs{z}^{\sigma}} \paren*{\frac{2}{\e'}}^{\paren*{\frac{m}{d}}^{n}} }^{d^n} \to 0
    \end{equation}
    as $n \to \infty$.
    Set $R \coloneqq \max\set{R_1, R_2, \frac{1}{1-C}}$ and fix any $z_0 \in D(0,R)^{c}$.
    By \eqref{eq:basin-at-infinity-11}, there exists $N_1 \in \NN $ such that for all $n \geq N_1$,
    \begin{equation}\label{eq:basin-at-infinity-N1}
        \frac{C_2^{d^n}}{\abs{z_0}^{\hat{q}d^n}} < \frac{C_1^{d^n}}{\abs{z_0}^{\sigma d^n}}.
    \end{equation}
    Moreover, if we set $\e' \coloneqq \min\set{\e, \delta_b, \delta_r, \delta_g,1}$, then \eqref{eq:basin-at-infinity-12} yields $N_2 \in \NN$ such that for all $n \geq N_2$,
    \begin{equation}\label{eq:basin-at-infinity-N2}
        \frac{C_1^{d^n}}{\abs{z_0}^{\sigma d^n}} < \paren*{\frac{\e'}{2}}^{m^n}.
    \end{equation}
    Set $N \coloneqq \max\set{N_1, N_2, M}$, and fix any $n \geq N$.

    By \eqref{eq:basin-at-infinity-N1} and \eqref{eq:basin-at-infinity-N2}, there exists $0 < \delta < \frac{\e'}{2} < \min\set{\e, \delta_b, \delta_r, \delta_g, 1}$ such that
    \begin{equation*}
        \frac{C_2^{d^n}}{\abs{z_0}^{\hat{q}d^n}} < \delta^{m^n} < \frac{C_1^{d^n}}{\abs{z_0}^{\sigma d^n}}.
    \end{equation*}
    By \eqref{eq:basin-at-infinity-C1-a}, \eqref{eq:basin-at-infinity-C1-b}, \eqref{eq:basin-at-infinity-C2-a}, and \eqref{eq:basin-at-infinity-C2-b} respectively, it follows that
    \begin{align}
        \delta^{m^n} (1+C)^{\frac{d^{n-1} - 1}{d - 1} m\hat{q}} \abs{z_0}^{d^{n-1} m\hat{q}}
            &\leq \paren*{\frac{C}{2B_{+}(\ell-d)}}^{m^n},
            \label{eq:basin-at-infinity-cond1}\\
        \delta^{m^n} (1+C)^{\frac{d^n - 1}{d - 1} q} \abs{z_0}^{d^n q}
            &\leq \frac{B_{-}}{2(\ell-d)B_{+}},
            \label{eq:basin-at-infinity-cond2}\\
        \delta^{m^n} (1-C)^{\frac{d^n - 1}{d - 1} \hat{q}} \abs{z_0}^{d^n \hat{q}}
            &\geq \paren*{\frac{2(C_f + 1)}{\alpha B_{-}}}^{\frac{1}{\hat{\nu}}}
                \paren*{\frac{1+C}{1-C}}^{d\frac{d^n-1}{\hat{\nu}(d-1)}},
            \quad\text{and} \label{eq:basin-at-infinity-cond3}\\
        \delta^{m^n} (1-C)^{\frac{d^n - 1}{d - 1} \hat{q}} \abs{z_0}^{d^n \hat{q}}
            &\geq \frac{2(\ell-d)B_{+}}{B_{-}},
            \label{eq:basin-at-infinity-cond4}
    \end{align}
    where we used that $\sigma = \max\set*{q, \frac{m}{d} \hat{q}}$ and $\abs{z_0} \geq 1$.

    In the following, for $w \in D(0,\delta_g)$, we denote the $k$th iterate of $(z_0, w)$ by $(z_k,w_k) = F^k(z_0,w)$.
    The preparations above yield a growth estimate for the first $n$ iterates as described in the following claim.
    \begin{claim}\label{claim:basin-at-infinity-1}
        For all $k \in \set{0, \dots, n}$ and all $w \in \overline{D(0,\delta)}$, it holds that
        \begin{equation}\label{eq:basin-at-infinity-13}
            (1 - C)^{\frac{d^k - 1}{d - 1}} \abs{z_0}^{d^k} \leq \abs{z_k} \leq (1 + C)^{\frac{d^k - 1}{d - 1}} \abs{z_0}^{d^k}.
        \end{equation}
    \end{claim}
    \begin{subproof}
        We proceed by induction.
        Fix any $w \in \overline{D(0,\delta)}$.
        The base case, $k = 0$, is evident.
        Assume that the statement is true for some $0 \leq k \leq n - 1$.
        In particular, it holds that $\abs{z_k} \geq \abs{z_0} \geq R$ since we assumed that $R (1-C) \geq 1$.
        Furthermore,
        \begin{equation*}
            \abs{\hat{r}(z_k, w_k)} = \abs*{\sum_{(\nu,\mu) \in \mathcal{R}} b_{\mu}(w_k) w_k^{\nu} z_k^{\mu} }  
            \leq \sum_{(\nu,\mu) \in \mathcal{R}} \abs{b_{\mu}(w_k)} \abs{w_k}^{\nu} \abs{z_k}^{\mu} 
            \leq B_{+} \sum_{(\nu,\mu) \in \mathcal{R}} \abs{w_k}^{\nu} \abs{z_k}^{\mu}.
        \end{equation*}
        Then, for any $(\nu,\mu)\in\mathcal{R}$,
        \begin{align*}
            \abs{w_k}^{\nu} \abs{z_k}^{\mu} &\leq \abs{w_n}^{\frac{\nu}{m^{n-k}}} (1+C)^{\frac{d^k - 1}{d - 1}\mu} \abs{z_0}^{d^k\mu} \\
            &= \paren*{\abs{w_n} (1+C)^{\frac{d^k - 1}{d - 1}m^{n-k}\frac{\mu}{\nu}} \abs{z_0}^{d^km^{n-k}\frac{\mu}{\nu}}}^{\frac{\nu}{m^{n-k}}}
        \end{align*}
        and using \eqref{eq:basin-at-infinity-cond1},
        \begin{align*}
            \abs{w_n} (1+C)^{\frac{d^k - 1}{d - 1}m^{n-k}\frac{\mu}{\nu}} \abs{z_0}^{d^km^{n-k}\frac{\mu}{\nu}} &\leq \delta^{m^n} (1+C)^{\frac{d^{n-1} - 1}{d - 1}m\frac{\mu}{\nu}} \abs{z_0}^{d^{n-1} m \frac{\mu}{\nu}} \\
            &\leq \delta^{m^n} (1+C)^{\frac{d^{n-1} - 1}{d - 1}m \hat{q}} \abs{z_0}^{d^{n-1} m \hat{q}} \\
            &\leq \paren*{\frac{C}{2B_{+}(\ell-d)}}^{m^n}.
        \end{align*}
        Thus, it follows from $\frac{C}{2B_{+}(\ell-d)} < 1$ that
        \begin{align*}
            \abs{w_k}^{\nu} \abs{z_k}^{\mu} &\leq \paren*{\frac{C}{2B_{+}(\ell-d)}}^{m^k \nu} \leq \frac{C}{2B_{+}(\ell-d)}.
        \end{align*}
        Together, the above inequalities yield
        \begin{align*}
            \abs{\hat{r}(z_k, w_k)} &\leq B_{+} \sum_{(\nu,\mu) \in \mathcal{R}} \abs{w_k}^{\nu} \abs{z_k}^{\mu} \leq B_{+} \sum_{(\nu,\mu) \in \mathcal{R}} \frac{C}{2B_{+}(\ell-d)} \leq \frac{C}{2}.
        \end{align*}
        Consequently, combining this inequality with \eqref{eq:basin-at-infinity-8} and \eqref{eq:basin-at-infinity-10}, we obtain that
        \begin{align*}
            \abs*{z_{k+1}} &= \abs{z_k}^d \abs*{1 + \sum_{j=0}^{d-1} a_j z_k^{j-d} + \sum_{j=0}^{d} r_j(w_k) z_k^{j-d} + \hat{r}(z_k,w_k)} \\
            &\leq \abs{z_k}^d \paren*{1 + \sum_{j=0}^{d-1}\abs{a_j} \abs{z_k}^{j-d} + \sum_{j=0}^{d} \abs{r_j(w_k)} \abs{z_k}^{j-d} + \abs*{\hat{r}(z_k,w_k)}} \\
            &\leq \abs{z_k}^d \paren*{1 + \sum_{j=0}^{d-1}\abs{a_j} R^{j-d} + \sum_{j=0}^{d} \abs{r_j(w_k)} R^{j-d} + \frac{C}{2}} \\
            &\leq \abs{z_k}^d \paren*{1 + \frac{C}{4} + \frac{C}{4} + \frac{C}{2}} = (1+C)\abs{z_k}^d
        \end{align*}
        and analogously,
        \begin{align*}
            \abs*{z_{k+1}} &\geq \abs{z_k}^d \paren*{1 - \sum_{j=0}^{d-1}\abs{a_j} \abs{z_k}^{j-d} - \sum_{j=0}^{d} \abs{r_j(w_k)} \abs{z_k}^{j-d} - \abs*{\hat{r}(z_k,w_k)}} \\
            &\geq \abs{z_k}^d \paren*{1 - \frac{C}{4} - \frac{C}{4} - \frac{C}{2}} = (1-C)\abs{z_k}^d.
        \end{align*}
        Together with the induction hypothesis, we obtain
        \begin{align*}
            \abs*{z_{k+1}} &\leq (1+C)\abs{z_k}^d \leq (1+C) \paren*{(1 + C)^{\frac{d^k - 1}{d - 1}} \abs{z_0}^{d^k}}^d = (1 + C)^{\frac{d^{k+1} - 1}{d - 1}} \abs{z_0}^{d^{k+1}}
        \end{align*}
        and similarly,
        \begin{equation*}
            \abs{z_{k+1}} \geq (1 - C)^{\frac{d^{k+1} - 1}{d - 1}} \abs{z_0}^{d^{k+1}},
        \end{equation*}
        which completes the induction.
    \end{subproof}
    \textit{Proof of Lemma~{\ref{lemma:sublemma-classification}} (continued).}
    \Cref{claim:basin-at-infinity-1} establishes the behavior for the first $n$ iterates.
    In particular, it yields property~\ref{prop:zn-nonzero} of \Cref{lemma:sublemma-classification}.
    
    We now turn to the $(n+1)$th iterate, whose behavior is controlled by one dominating term, given by the exponents $(\hat{\nu}, \hat{\mu})$.
    \begin{claim}\label{claim:basin-at-infinity-2}
        For any $w \in \partial D(0,\delta)$, it holds that
        \begin{equation}\label{eq:basin-at-infinity-14}
            \abs*{f(z_n) + r(z_n,w_n) - z_n^d \hat{r}(z_n,w_n)} = \abs*{f(z_n) + \sum_{j = 0}^{d} r_j(w_n) z_n^j} \leq \alpha \abs{z_n^d \hat{r}(z_n,w_n)}
        \end{equation}
        and that
        \begin{equation}\label{eq:basin-at-infinity-15}
            \hat{r}(z_n,w_n) = w_n^{\hat{\nu}} z_n^{\hat{\mu}} b_{\hat{\mu}}(w_n) ( 1 + \phi(z_n,w_n) )
        \end{equation}
        for a holomorphic function $\phi$ defined on $(\CC\setminus\set{0})\times (D(0,\delta_b)\setminus\set{0})$ satisfying $\abs{\phi(z_n,w_n)} \leq \frac{1}{2}$ for all $w \in \partial D(0,\delta)$.
    \end{claim}

    \begin{subproof}
        Fix any $w \in \partial D(0,\delta)$. 
        Using \Cref{claim:basin-at-infinity-1}, we estimate
        \begin{align*}
            \abs*{f(z_n) + \sum_{j = 0}^{d} r_j(w_n) z_n^j} &\leq C_f \abs{z_n}^d + \sum_{j = 0}^{d} \abs{r_j(w_n)} \abs{z_n}^j \\
            &\leq C_f \abs{z_n}^d + \abs{z_n}^d \\
            &\leq (C_f + 1) (1+C)^{d \frac{d^n - 1}{d - 1}} \abs{z_0}^{d^{n+1}}
        \end{align*}
        and 
        \begin{align*}
            \abs{z_n^d \hat{r}(z_n,w_n)} \geq \abs{\hat{r}(z_n,w_n)} (1-C)^{d \frac{d^n - 1}{d - 1}} \abs{z_0}^{d^{n+1}}.
        \end{align*}
        Hence, to prove \eqref{eq:basin-at-infinity-14}, it suffices to show that
        \begin{equation}\label{eq:basin-at-infinity-16}
            \abs{\hat{r}(z_n,w_n)} \geq \frac{C_f + 1}{\alpha} \paren*{ \frac{1+C}{1-C} }^{d \frac{d^n - 1}{d - 1}}.
        \end{equation}
        Furthermore, observe that
        \begin{align*}
            \hat{r}(z_n,w_n) &= \sum_{(\nu,\mu) \in \mathcal{R}} b_{\mu}(w_n) w_n^{\nu} z_n^{\mu} \\
            &= w_n^{\hat{\nu}} z_n^{\hat{\mu}} b_{\hat{\mu}}(w_n) \paren*{1 + \frac{1}{b_{\hat{\mu}}(w_n)}\sum_{(\nu,\mu) \in \mathcal{R}^{\star}} b_{\mu}(w_n) w_n^{\nu - \hat{\nu}} z_n^{\mu - \hat{\mu}}}
        \end{align*}
        where $\mathcal{R}^{\star} \coloneqq \mathcal{R} \setminus \set{(\hat{\nu}, \hat{\mu})}$, and define
        \begin{equation*}
            \phi(\zeta, \xi) \coloneqq \frac{1}{b_{\hat{\mu}}(\xi)}\sum_{(\nu,\mu) \in \mathcal{R}^{\star}} b_{\mu}(\xi) \xi^{\nu - \hat{\nu}} \zeta^{\mu - \hat{\mu}}
        \end{equation*}
        for $\zeta \in \CC\setminus\set{0}$ and $\xi \in D(0,\delta_b)\setminus\set{0}$.

        To complete the proof of the claim, it suffices to show that for all $(\nu,\mu) \in \mathcal{R}^{\star}$,
        \begin{equation}\label{eq:basin-at-infinity-17}
            \abs{w_n}^{\nu - \hat{\nu}} \abs{z_n}^{\mu - \hat{\mu}} \leq \frac{B_{-}}{2(\ell-d)B_{+}}
        \end{equation}
        and
        \begin{equation}\label{eq:basin-at-infinity-18}
            \abs{w_n}^{\hat{\nu}} \abs{z_n}^{\hat{\mu}} \geq \frac{2(C_f + 1)}{\alpha B_{-}} \paren*{ \frac{1+C}{1-C} }^{d \frac{d^n - 1}{d - 1}}.
        \end{equation}
        Indeed, inequality \eqref{eq:basin-at-infinity-17} then yields
        \begin{equation*}
            \abs{\phi(z_n,w_n)} \leq \frac{B_{+}}{B_{-}}\sum_{(\nu,\mu) \in \mathcal{R}^{\star}} \abs{w_n}^{\nu - \hat{\nu}} \abs{z_n}^{\mu - \hat{\mu}} \leq \frac{1}{2},
        \end{equation*}
        which is the claim about the factorization in \eqref{eq:basin-at-infinity-15}, and together with \eqref{eq:basin-at-infinity-18}, it also follows that
        \begin{align*}
            \abs{\hat{r}(z_n,w_n)} &\geq \abs{w_n}^{\hat{\nu}} \abs{z_n}^{\hat{\mu}} B_{-} \paren*{1 - \abs{\phi(z_n, w_n)}} \\
            &\geq \frac{B_{-}}{2} \abs{w_n}^{\hat{\nu}} \abs{z_n}^{\hat{\mu}} \\
            &\geq \frac{C_f + 1}{\alpha} \paren*{ \frac{1+C}{1-C} }^{d \frac{d^n - 1}{d - 1}},
        \end{align*}
        which is \eqref{eq:basin-at-infinity-16}, thus proving \eqref{eq:basin-at-infinity-14}.

        To prove that \eqref{eq:basin-at-infinity-17} holds, we consider two cases based on the partition of $\mathcal{R}$.
        First, consider $(\nu,\mu) \in \mathcal{R}_{1}$, that is, $\frac{\mu}{\nu} = \frac{\hat{\mu}}{\hat{\nu}}$ and $\hat{\mu} > \mu$.
        In particular, $\hat{\nu} > \nu$ and $\hat{q} = \frac{\hat{\mu}}{\hat{\nu}} = \frac{\hat{\mu} - \mu}{\hat{\nu} - \nu}$.
        Then,
        \begin{align*}
            \abs{w_n}^{\nu - \hat{\nu}} \abs{z_n}^{\mu - \hat{\mu}} &= \frac{1}{\abs{w_n}^{\hat{\nu} - \nu} \abs{z_n}^{\hat{\mu} - \mu}} = \paren*{\frac{1}{\abs{w_n} \abs{z_n}^{\hat{q}}}}^{\hat{\nu} - \nu}.
        \end{align*}
        Using \eqref{eq:basin-at-infinity-cond4}, we obtain
        \begin{align*}
            \abs{w_n} \abs{z_n}^{\hat{q}} &\geq \delta^{m^n} (1-C)^{\frac{d^n - 1}{d - 1} \hat{q}} \abs{z_0}^{d^n \hat{q}} \geq \frac{2 (\ell-d)B_{+}}{B_{-}},
        \end{align*}
        and thus,
        \begin{align*}
            \abs{w_n}^{\nu - \hat{\nu}} \abs{z_n}^{\mu - \hat{\mu}} &\leq \paren*{\frac{B_{-}}{2(\ell-d)B_{+}}}^{\hat{\nu} - \nu} \leq \frac{B_{-}}{2(\ell-d)B_{+}}.
        \end{align*}
        Next, let $(\nu,\mu) \in \mathcal{R}_{2}$, that is, $\frac{\mu}{\nu} < \frac{\hat{\mu}}{\hat{\nu}}$.
        Then,
        \begin{align*}
            \abs{w_n}^{\nu - \hat{\nu}} \abs{z_n}^{\mu - \hat{\mu}} = \frac{\paren*{\abs{w_n} \abs{z_n}^{\frac{\mu}{\nu}}}^{\nu}}{\paren*{\abs{w_n} \abs{z_n}^{\hat{q}}}^{\hat{\nu}}}
            \leq \frac{\paren*{\abs{w_n} \abs{z_n}^{q}}^{\nu}}{\paren*{\abs{w_n} \abs{z_n}^{\hat{q}}}^{\hat{\nu}}}
            \leq \frac{\paren*{\delta^{m^n} (1+C)^{q \frac{d^n - 1}{d - 1}} \abs{z_0}^{qd^n} }^{\nu}}{\paren*{\delta^{m^n}(1-C)^{\hat{q} \frac{d^n - 1}{d - 1}} \abs{z_0}^{\hat{q} d^n} }^{\hat{\nu}}}.
        \end{align*}
        Since \eqref{eq:basin-at-infinity-cond2} yields
        \begin{equation*}
            \paren*{\delta^{m^n} (1+C)^{q \frac{d^n - 1}{d - 1}} \abs{z_0}^{qd^n} }^{\nu} \leq \paren*{\frac{B_{-}}{2(\ell-d)B_{+}}}^{\nu} \leq \frac{B_{-}}{2(\ell-d)B_{+}},
        \end{equation*}
        and \eqref{eq:basin-at-infinity-cond4} yields
        \begin{align*}
            \paren*{\delta^{m^n}(1-C)^{\hat{q} \frac{d^n - 1}{d - 1}} \abs{z_0}^{\hat{q} d^n} }^{\hat{\nu}} \geq \paren*{\frac{2(\ell-d)B_{+}}{B_{-}}}^{\hat{\nu}}\geq 1,
        \end{align*}
        it follows that
        \begin{align*}
            \abs{w_n}^{\nu - \hat{\nu}} \abs{z_n}^{\mu - \hat{\mu}} \leq \frac{B_{-}}{2(\ell-d)B_{+}}.
        \end{align*}
        Next, we show that \eqref{eq:basin-at-infinity-18} holds.
        Note that by \Cref{claim:basin-at-infinity-1} and by \eqref{eq:basin-at-infinity-cond3},
        \begin{align*}
            \abs{w_n} \abs{z_n}^{\hat{q}} &\geq \delta^{m^n} (1-C)^{\frac{d^{n} - 1}{d - 1} \hat{q}}\abs{z_0}^{d^n \hat{q}}
            \geq \paren*{\frac{2(C_f + 1)}{\alpha B_{-}}  }^{\frac{1}{\hat{\nu}}} \paren*{\frac{1+C}{1-C}}^{d\frac{d^n-1}{\hat{\nu}(d-1)}}
        \end{align*}
        which implies that
        \begin{align*}
            \abs{w_n}^{\hat{\nu}} \abs{z_n}^{\hat{\mu}} &= \paren*{ \abs{w_n} \abs{z_n}^{\hat{q}}}^{\hat{\nu}} \geq \frac{2(C_f + 1)}{\alpha B_{-}}  \paren*{\frac{1+C}{1-C}}^{d\frac{d^n-1}{d-1}},
        \end{align*}
        which is \eqref{eq:basin-at-infinity-18}.
        This concludes the proof of \eqref{eq:basin-at-infinity-14} and \eqref{eq:basin-at-infinity-15}.
    \end{subproof}
    \textit{Proof of Lemma~{\ref{lemma:sublemma-classification}} (continued).}
    Since
    \begin{equation*}
        r(z_n,w_n) - z_n^d \hat{r}(z_n,w_n) = \sum_{j = 0}^{d} r_j(w_n) z_n^j,
    \end{equation*}
    this finishes the proof of \Cref{lemma:sublemma-classification} in the superattracting case.

    The proof for $m = 1$ follows the same structure without any conceptual departures. 
    We give a brief overview.
    In this case, $g(w) = \lambda w$ for some $\lambda \in \DD\setminus\set{0}$. 
    Thus, $w_n = \lambda^n w$, and, in particular, $w_k = \lambda^{k-n} w_n = \lambda^{k} w$ for any $0 \leq k \leq n$.
    Consequently, the main adjustment is that we use $\abs{\lambda}^n \delta$ instead of $\delta^{m^n}$.
    Note that now $\sigma = \max\set*{q, \frac{1}{d}\hat{q}}$.
    More precisely, \eqref{eq:basin-at-infinity-C1-a} is replaced by
    \begin{equation}\label{eq:basin-at-infinity-C1-a-attr}
        C_1^{d^n} \leq \abs{\lambda}^n \frac{C}{2B_{+}(\ell-d)}\, (1+C)^{-\frac{d^{n-1} - 1}{d - 1} \hat{q}}.
    \end{equation}
    Such a constant $0 < C_1 < 1$ exists since $\abs{\lambda}^{\frac{d^n-1}{d-1}} \leq \abs{\lambda}^n \leq 1$.
    Moreover, \eqref{eq:basin-at-infinity-12} is replaced by
    \begin{equation}\label{eq:basin-at-infinity-12-attr}
        \frac{\paren*{\frac{C_1}{\abs{z}^{\sigma}}}^{d^n}}{\abs{\lambda}^n\frac{\e'}{2}} = \paren*{ \frac{C_1}{\abs{z}^{\sigma}} \paren*{\frac{2}{\e'\abs{\lambda}^n}}^{\frac{1}{d^n}} }^{d^n} \to 0
    \end{equation}
    as $n \to \infty$, which holds for all $z \in D(0,R_2)^c$ since $\frac{C_1}{\abs{z}^{\sigma}} < 1$ and $\paren*{\abs{\lambda}^{-n}}^{\frac{1}{d^n}} \to 1$.
    Accordingly, \eqref{eq:basin-at-infinity-N2} becomes $\frac{C_1^{d^n}}{\abs{z_0}^{\sigma d^n}} < \abs{\lambda}^n \frac{\e'}{2}$, and we obtain $0 < \delta < \frac{\e'}{2}$ such that
    \begin{equation*}
        \frac{C_2^{d^n}}{\abs{z_0}^{\hat{q}d^n}} < \abs{\lambda}^n \delta < \frac{C_1^{d^n}}{\abs{z_0}^{\sigma d^n}}.
    \end{equation*}
    Since $\sigma d^n \geq \hat{q} d^{n-1}$ and $\abs{z_0} \geq 1$, \eqref{eq:basin-at-infinity-C1-a-attr} yields that \eqref{eq:basin-at-infinity-cond1} changes to
    \begin{equation}\label{eq:basin-at-infinity-19}
        \abs{\lambda}^n \delta (1+C)^{\frac{d^{n-1} - 1}{d - 1} \hat{q}} \abs{z_0}^{d^{n-1} \hat{q}} \leq \abs{\lambda}^n\frac{C}{2B_{+}(\ell-d)}.
    \end{equation}
    The analogues of conditions~\eqref{eq:basin-at-infinity-cond2}--\eqref{eq:basin-at-infinity-cond4}, restated with $\abs{\lambda}^n \delta$ in place of $\delta^{m^n}$, follow from \eqref{eq:basin-at-infinity-C1-b}--\eqref{eq:basin-at-infinity-C2-b} as before.

    For \Cref{claim:basin-at-infinity-1}, since $\abs{w_k} = \abs{\lambda}^{k} \abs{w} \leq \delta$, the key estimate $\abs{w_k}^{\nu} \abs{z_k}^{\mu} \leq \frac{C}{2B_{+}(\ell-d)}$ follows from
    \begin{equation*}
        \abs{w_k}^{\nu} \abs{z_k}^{\mu} \leq \paren*{\delta (1+C)^{\frac{d^k - 1}{d - 1}\frac{\mu}{\nu}} \abs{z_0}^{d^k\frac{\mu}{\nu}}}^{\nu}.
    \end{equation*}    
    For $k \leq n-1$ and $\frac{\mu}{\nu} \leq \hat{q}$, the adapted condition~\eqref{eq:basin-at-infinity-19} yields 
    \begin{equation*}
        \paren*{\delta (1+C)^{\frac{d^k - 1}{d - 1}\frac{\mu}{\nu}} \abs{z_0}^{d^k\frac{\mu}{\nu}}}^{\nu} \leq \paren*{\frac{C}{2B_{+}(\ell-d)}}^{\nu} \leq \frac{C}{2B_{+}(\ell-d)},
    \end{equation*}
    and combining both gives the bound. 
    The rest of the inductive step is unchanged.

    The proof of \Cref{claim:basin-at-infinity-2} carries over verbatim, using $\abs{w_n} = \abs{\lambda}^n \delta$ for $w \in \partial D(0,\delta)$.
\end{proof}

\begin{lemma}
    Under the hypotheses of \Cref{thm:basin-inf-classification}, non-bulging for the basin at infinity implies nowhere bulging.
\end{lemma}

\begin{proof}
    Fix any $\zeta$ in the basin at infinity of $f$.
    It suffices to show that $(\zeta,0)$ belongs to the Julia set of $F$.
    Suppose, to the contrary, that $(\zeta, 0)$ is in the Fatou set of $F$.
    Since both $\F(F)$ and the basin at infinity of $f$ are open, there exist $\delta_1 > 0$ and $\delta_2 > 0$ such that, with $D_1 \coloneqq D(\zeta,\delta_1)$ and $D_2 \coloneqq D(0,\delta_2)$, the closure $\overline{D_1}$ lies in the basin at infinity of $f$ and $\overline{D_1 \times D_2} \subset \F(F)$.
    Since $F^n(\zeta,0) = (f^n(\zeta),0) \to \infty $, the only possible limit behavior is that the iterates diverge uniformly to infinity in the $z$-coordinate on $D_1 \times D_2$.

    By \Cref{lemma:basin-inf-classification-oneway}, non-bulging implies $m < d < \ell$.
    Hence, by the proof of \Cref{lemma:basin-at-inf-2}, there exists $R > 0$ such that $D(0,R)^c$ lies in the basin at infinity of $f$, and such that for every $ z \in D(0,R)^c$, $\e > 0$, and $M \in \NN$, there exist $\widetilde{w} \in D(0,\e)$ and $n \geq M$ satisfying \eqref{eq:basin-at-infinity-7}, that is,
    \begin{equation}\label{eq:basin-at-infinity-7-restate}
        \pi_z F^{n}(z,\widetilde{w}) = 0.
    \end{equation}
    Since $\overline{D_1}$ is a compact subset of the basin at infinity of $f$, there exists $N \in \NN$ with $f^N(D_1) \subset D(0,R)^c$.
    Define $z_0 \coloneqq f^N(\zeta)$ and
    \begin{equation*}
        G(z,w) \coloneqq (\phi(z,w), \psi(w)) \coloneqq F^N(z,w).
    \end{equation*}
    Since $G$ is a skew product,
    \begin{equation*}
        G(D_1 \times D_2) = \bigcup_{w \in D_2} \phi(D_1 \times \set{w}) \times \set{\psi(w)}.
    \end{equation*}
    The function $\phi(\cdot,0) - z_0 = f^N(\cdot) - f^N(\zeta)$ is not identically zero and vanishes at $\zeta \in D_1$.
    By the continuity of zeros (see, e.g., \cite[Lemma~3.2]{Range1986}), after shrinking $\delta_2$
    if necessary,
    \begin{equation*}
        z_0 \in \phi(D_1 \times \set{w}) 
    \end{equation*}
    for every $w \in D_2 = D(0,\delta_2)$.
    In particular,
    \begin{equation*}
        \set{z_0} \times \psi(D_2) \subset G(D_1 \times D_2).
    \end{equation*}
    Fix any $\e > 0$ such that $D(0,\e) \subset \psi(D_2)$ and any $M \in \NN$.
    Choose $\widetilde{w} \in D(0,\e)$ and $n \geq M$ according to \eqref{eq:basin-at-infinity-7-restate}.
    Then there exists $(\hat{z}_0,\hat{w}) \in D_1 \times D_2$ with $F^N(\hat{z}_0, \hat{w}) = (z_0, \widetilde{w})$.
    Hence, by \eqref{eq:basin-at-infinity-7-restate}, 
    \begin{equation*}
        \pi_z F^{N+n}(\hat{z}_0, \hat{w}) = \pi_z F^n(z_0, \widetilde{w}) = 0.
    \end{equation*}
    Since $M$ was arbitrary, the iterates cannot diverge uniformly to infinity on $D_1 \times D_2$, contradicting the uniform divergence established above.
\end{proof}

\section{Nowhere Bulging}\label{sec:nowhere-bulging}
We now prove \Cref{thm:nowhere-bulging}, which strengthens \cite[Theorem 1.1]{Potthink2026} from non-bulging to nowhere bulging as defined in \Cref{def:bulging}.

\nowhere*

The proof of \Cref{thm:nowhere-bulging} is a generalization of the proof of \cite[Theorem 1.1]{Potthink2026} obtained by carrying out the same construction simultaneously for countably many points.

\begin{proof}[Proof of Theorem~{\ref{thm:nowhere-bulging}}]
    Let $Q$ be a countable and dense subset of $U$ such that $f^n(q) \neq f^m(p)$ for all $n,m \in \NN$ and for all $p \neq q$ in $Q$.
    This is possible since for any fixed $q \in U$, the set $\bigcup_{n,m \in\NN} f^{-n}(f^m(q))$ is countable.
    Hence, one can build $Q$ inductively by choosing each next point outside the countable bad set determined by the previous ones and arbitrarily close to a prescribed point of a countable dense subset of $U$.
    Furthermore, $f^n(q) \neq f^m(q)$ for all $n \neq m$ and for all $q \in Q$ since orbits of points in $U$ accumulate at infinity and therefore cannot be preperiodic.
    Let $(q_m)_{m\in\NN}$ be an enumeration of $Q$.
    
    Given any entire function $h$, we consider the skew product
    \begin{equation*}
        F_h (z,w) = (f(z) + wh(z), g(w)).
    \end{equation*}
    Observe that $ g(0) = 0 $ implies that $\set{w=0}$ is an invariant fiber of $F_h$ independent of our choice of $h$.
    The following lemma provides an entire function $h$ such that arbitrarily close to any point $(q,0)  \in Q \times \set{0}$, we find points which are mapped close to the origin after arbitrarily large numbers of iterations and points which are mapped close to infinity after arbitrarily large numbers of iterations.

    \begin{lemma}\label{lemma:nowhere-bulging-1}
        There exist an entire function $h $, a subsequence of iterates $(f^{n_k})_{k\in\NN}$, and two sequences $(w_k)_{k\in\NN}$ and $(w'_k)_{k\in\NN}$ of nonzero complex numbers converging to $0$ such that the skew product $F_h$ satisfies for all $k \geq 1$ and all $m \in \set{0,\dots, k}$,
        \begin{equation}\label{eq:nowhere-bulging-1}
            \norm{F_h^{n_k+1}(q_m, w_k) } < 1,
        \end{equation}
        and
        \begin{equation}\label{eq:nowhere-bulging-2}
            \norm{F_h^{n_k+1}(q_m, w'_k) } > k.
        \end{equation}
    \end{lemma}
    \textit{Proof of Theorem~{\ref{thm:nowhere-bulging}} (continued).}
    We now prove that $U \times \set{0}$ is contained in the Julia set $\J(F)$ of $F \coloneqq F_h$, where $h$ is as provided by \Cref{lemma:nowhere-bulging-1}.
    
    Assume, to the contrary, that there exists $z_0 \in U$ for which $(z_0, 0)$ belongs to the Fatou set $ \F(F)$ of $F$. 
    Then, the iterates of $F$ form a normal family in a neighborhood of $(z_0, 0)$. 
    In particular, one can find a compact neighborhood $C$ of $(z_0,0)$ and a subsequence $(F^{n_{k(j)} + 1})_{j\in\NN}$ of iterates such that this subsequence either diverges uniformly to $\infty$ on $C$ or converges uniformly to a finite limit map on $C$.

    First, assume that the subsequence diverges uniformly to infinity. 
    In particular, there exists $J \in\NN$ such that for all $j \geq J$ and $(z,w) \in C$,
    \begin{equation*}
        \norm{F^{n_{k(j)} + 1}(z,w)} \geq 1.
    \end{equation*}
    Since $Q$ is dense in $U$ and $C$ is a neighborhood of $(z_0, 0)$, there are $\kappa \in \NN$ and $0 < \rho < 1$ such that $\set{q_\kappa} \times D(0,\rho) \subset C$.
    Let $J' \in \NN$ be large enough that for all $j \geq J'$, we have $\abs{w_{k(j)}} < \rho$, and therefore, $(q_\kappa, w_{k(j)}) \in C$.
    Lastly, let $J''\in \NN$ be such that $k(j) \geq \kappa$ for all $j \geq J''$.
    Applying inequality \eqref{eq:nowhere-bulging-1}, we obtain for all $j \geq \max\set{J, J', J''} $,
    \begin{equation*}
        1 \leq \norm{F^{n_{k(j)} + 1}(q_\kappa, w_{k(j)})} < 1,
    \end{equation*}
    which is a contradiction.

    Now, assume that the subsequence of iterates has a finite limit map $\Psi\colon C \to \CC^2$, i.e., $F^{n_{k(j)} + 1}$ converges to $\Psi$ uniformly on $C$ as $j \to \infty$. Then there exists $J \in \NN$ with
    \begin{equation}\label{eq:nowhere-bulging-3}
        \norm{F^{n_{k(j)} + 1}(z,w) - \Psi(z,w)} \leq 1
    \end{equation}
    for all $j \geq J$ and $(z,w) \in C$. As before, we find $\kappa$, $J'$, and $ J'' \in \NN$ with $(q_\kappa, w_{k(j)}') \in C$ for all $j \geq J'$, and with $k(j) \geq \kappa$ for all $j \geq J''$.
    Let $A \coloneqq \max_{(z,w) \in C} \norm{\Psi(z,w)} < \infty$ and choose $J'''\in\NN$ such that $k(j) > A + 1 $ for all $j \geq J'''$.
    Applying \eqref{eq:nowhere-bulging-2} and the reverse triangle inequality, we obtain
    \begin{align*}
        \lVert F^{n_{k(j)} + 1}(q_\kappa,w_{k(j)}') - \Psi(q_\kappa,w_{k(j)}') \rVert &\geq \left\lvert \lVert F^{n_{k(j)} + 1}(q_\kappa,w_{k(j)}') \rVert - \lVert \Psi(q_\kappa,w_{k(j)}') \rVert \right\rvert \\
        &\geq k(j) - A \\
        &> A + 1 - A = 1
    \end{align*}
    for all $j \geq \max\set{J,J',J'', J'''}$, in contradiction to \eqref{eq:nowhere-bulging-3}. 
    Therefore, $(z_0,0)$ does not belong to the Fatou set of $F$, which completes the proof.
\end{proof}

We now prove \Cref{lemma:nowhere-bulging-1}.

\begin{proof}[Proof of Lemma~{\ref{lemma:nowhere-bulging-1}}]
    By the hypotheses, there exist $z_0 \in U$ and a subsequence of iterates $(f^{n_k})_{k\in\NN}$ with $(n_k)_{k\in\NN}$ strictly increasing such that $f^{n_k}(z_0) \to \infty$ as $k \to \infty$.
    By normality on $U$, up to considering a subsequence, it follows that
    \begin{equation}\label{eq:nowhere-bulging-4}
        \abs{f^{n_k}(q_m)} \to \infty
    \end{equation}
    as $k \to \infty$ for all $m \in\NN$.

    The function $h$ will be defined as the limit of a sequence of entire functions which are inductively constructed using Runge's Theorem~\ref{thm:runge-1}. 
    To that end, we require several sequences of sets which are described in the following claim and shown in Figure~\ref{fig:nowhere-bulging-sets}.

    \begin{claim}\label{claim:nowhere-bulging}
        Up to passing to a subsequence of $(f^{n_k})_{k\in\NN}$, there exist a sequence of nonnegative numbers $(R_k)_{k\in\NN}$ and a sequence of positive numbers $(\rho_k)_{k\in\NN}$ such that letting $\widetilde{H}_k \coloneqq \overline{D(0,R_k)}$ and $\Delta^j_{m,k} \coloneqq \overline{D(f^j(q_m), \rho_k)}$, the following hold for every $k\in\NN$:
        \begin{enumerate}[label=(\arabic*)]
            \item $R_k \geq k$,
            \item the sets $\Delta^j_{m,k}$, $m = 0,\dots,k$, $j = 0, \dots, n_k$, are pairwise disjoint,
            \item we have the inclusion
            \begin{equation*}
                \widetilde{H}_k \cup \bigcup_{m = 0}^k \bigcup_{j = 0}^{n_k} \Delta^j_{m,k} \subset \widetilde{H}_{k+1},
            \end{equation*}
            \item $\widetilde{H}_k$ and $ \Delta^{n_k}_{0,k} \cup \dots \cup \Delta^{n_k}_{k,k}$ are disjoint,
            \item for all $z, z' \in \Delta^j_{m,k}$, $m = 0,\dots,k$, and $j = 0, \dots, n_k$,
            \begin{equation}\label{eq:nowhere-bulging-5}
                \abs{f(z) - f(z')} < \frac{1}{4}.
            \end{equation}
        \end{enumerate}
    \end{claim}

    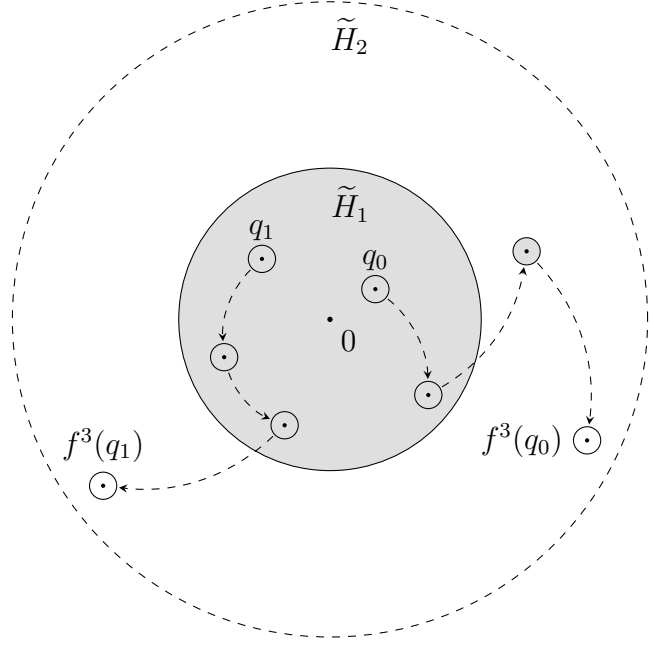
\begin{figure}[ht]
        \centering
        \begin{tikzpicture}[
            scale=1,
            orbit/.style={->, >=stealth, dashed, shorten >=6pt, shorten <=6pt},
        ]
            \def\rad{0.18}

            \draw[dashed] (0,0) circle (4.2);
            \node[left] at (80:3.8) {$\widetilde{H}_{2}$};

            \filldraw[fill=gray!25] (0,0) circle (2);
            \node at (80:1.55) {$\widetilde{H}_1$};
            \fill (0,0) circle (1pt) node[below right] {$0$};

            \coordinate (a0) at (0.6,0.4);
            \coordinate (a1) at (1.3,-1.0);
            \coordinate (a2) at (2.6,0.9);
            \coordinate (a3) at (3.4,-1.6);

            \coordinate (b0) at (-0.9,0.8);
            \coordinate (b1) at (-1.4,-0.5);
            \coordinate (b2) at (-0.6,-1.4);
            \coordinate (b3) at (-3.0,-2.2);

            \foreach \p in {a0,a1,a2,b0,b1,b2}
                \draw[draw, fill=gray!25] (\p) circle (\rad);

            \foreach \p in {a3,b3}
                \draw[draw, fill=white] (\p) circle (\rad);

            \foreach \p in {a0,a1,a2,a3,b0,b1,b2,b3}
                \fill (\p) circle (0.8pt);

            \draw[orbit] (a0) to[bend left=25]  (a1);
            \draw[orbit] (a1) to[bend right=25] (a2);
            \draw[orbit] (a2) to[bend left=25]  (a3);
            \draw[orbit] (b0) to[bend right=25] (b1);
            \draw[orbit] (b1) to[bend right=25] (b2);
            \draw[orbit] (b2) to[bend left=25]  (b3);

            \node[above=3pt] at (a0) {$q_0$};
            \node[above=3pt] at (b0) {$q_1$};
            \node[left=5pt] at (a3) {$f^{3}(q_0)$};
            \node[above=5pt]  at (b3) {$f^{3}(q_1)$};
        \end{tikzpicture}
        \caption{The sets from Claim~\ref{claim:nowhere-bulging} for $k = 1$ and $n_k = 3$.
        The disks $\Delta^{n_k}_{m,k}$ around $f^{n_k}(q_m)$ lie outside $\widetilde{H}_k$ but inside $\widetilde{H}_{k+1}$.}
        \label{fig:nowhere-bulging-sets}
        \end{figure}

    \begin{subproof}[Proof of Claim~{\ref{claim:nowhere-bulging}}]
        The proof is constructive. 
        We redefine $n_0 \coloneqq 0$.
        Let $R_0 \coloneqq 0 $ so that $\widetilde{H}_0 = \set{0}$. 
        Up to conjugating by a translation, we may assume that $q_0 \neq 0$.
        Choose any $\rho_0 > 0$ with $ 0 \not\in \overline{D(q_0, \rho_0)} = \Delta^{0}_{0,0}$ and so small that \eqref{eq:nowhere-bulging-5} holds.

        Assuming that $R_0, \dots, R_k$, $n_0, \dots, n_k$, and $\rho_0, \dots, \rho_k$ have been constructed in accordance with the required properties, we construct the next elements.

        Since $\widetilde{H}_k \cup \bigcup_{m = 0}^k \bigcup_{j = 0}^{n_k} \Delta^j_{m,k}$ is bounded, there exists $R_{k+1} > k + 1$ with
        \begin{equation*}
            \widetilde{H}_k \cup \bigcup_{m = 0}^k \bigcup_{j = 0}^{n_k} \Delta^j_{m,k} \subset \overline{D(0, R_{k+1})} = \widetilde{H}_{k+1}.
        \end{equation*}
        By \eqref{eq:nowhere-bulging-4}, there exists $\ell \geq k+1$ such that
        \begin{equation}\label{eq:nowhere-bulging-6}
            f^{n_\ell}(q_m) \in \CC \setminus \widetilde{H}_{k+1}
        \end{equation}
        for all $ m \in \set{ 0,\dots, k+1 }$. 
        Up to passing to a subsequence of iterates, we may assume that $\ell = k+1$.
        Since, by our assumption on $Q$, the points $f^j(q_m)$, $0 \leq j \leq n_{k+1}$ and $0 \leq m \leq k+1$, are pairwise distinct, there exists $\rho_{k+1} > 0$ such that all sets $\Delta^j_{m,k+1}$ are pairwise disjoint.
        By \eqref{eq:nowhere-bulging-6} and choosing a smaller $\rho_{k+1}$ if necessary, we may assume that $\widetilde{H}_{k+1}$ and $ \Delta^{n_{k+1}}_{0,{k+1}} \cup \dots \cup \Delta^{n_{k+1}}_{k+1,k+1}$ are disjoint, and, by uniform continuity on compact sets, that \eqref{eq:nowhere-bulging-5} holds. 
    \end{subproof}

    \textit{Proof of Lemma~{\ref{lemma:nowhere-bulging-1}} (continued).}
    For $k\in\NN$, define
    \begin{equation*}
        H_k \coloneqq \widetilde{H}_k \cup \bigcup_{m = 0}^k \bigcup_{j = 0}^{n_k - 1} \Delta^j_{m,k}.
    \end{equation*}
    Observe that each $H_k$ is compact with connected complement and that it is disjoint from the sets $ \Delta^{n_k}_{0,k}, \dots, \Delta^{n_k}_{k,k}$ by the previous claim. 
    Furthermore, $H_k \subset H_{k+1}$ and $\CC = \bigcup_{k \in \NN} H_k$.

    The following lemma provides the key step of the construction.
    The conditions are chosen to induce the desired behavior for selected orbits (properties~\ref{prop:nowhere-bulging-lemma-2-2} and \ref{prop:nowhere-bulging-lemma-2-3}), to allow the application of Runge's Theorem in the construction (properties~\ref{prop:nowhere-bulging-lemma-2-4} and \ref{prop:nowhere-bulging-lemma-2-5}), and to ensure that this behavior persists for the limit function (properties~\ref{prop:nowhere-bulging-lemma-2-6} and \ref{prop:nowhere-bulging-lemma-2-7}).

    \begin{lemma}\label{lemma:nowhere-bulging-2}
        There exist a sequence of entire functions $(h_k)_{k\in\NN}$, two sequences of nonzero complex numbers $(w_k)_{k\in\NN}$ and $(w'_k)_{k\in\NN}$, and two sequences of positive real numbers $(r_k)_{k\in\NN}$ and $(\delta_k)_{k\in\NN}$ such that if we define $\zeta_{k,m} \coloneqq \pi_z F_{h_{k-1}}^{n_k}(q_m, w_k)$ and $\zeta'_{k,m} \coloneqq \pi_z F_{h_{k-1}}^{n_k}(q_m, w'_k)$ together with $A^k_m \coloneqq \overline{D(\zeta_{k,m}, r_k)}$ and $B^k_m \coloneqq \overline{D(\zeta'_{k,m}, r_k)}$ for $m = 0, \dots, k$, then we have for all $k \geq 1$,
        \begin{enumerate}[label=(\roman*)]
            \item $\abs{w_k} \leq \frac{1}{k}$, $\abs{w'_k} \leq \frac{1}{k}$, $\delta_k < \frac{1}{4}$\label{prop:nowhere-bulging-lemma-2-1},
            \item $0 < \abs{g^{j}(w_k)} < \frac{1}{4}$ and $ 0 < \abs{g^{j}(w'_k)} < \frac{1}{4}$ for all $j \in \set{0, \dots, n_k + 1}$,
            \item for all $m \in \set{0, \dots, k}$,
            \begin{equation*}
                \norm*{h_k + \frac{1}{g^{n_k}(w_k)} f(\zeta_{k,m})}_{A_m^k} < \frac{1}{4},
            \end{equation*} \label{prop:nowhere-bulging-lemma-2-2}
            \item for all $m \in \set{0, \dots, k}$,
            \begin{equation*}
                \norm*{h_k + \frac{1}{g^{n_k}(w_k')} (f(\zeta'_{k,m}) - (k+1))}_{B_m^k} < \frac{1}{4},
            \end{equation*} \label{prop:nowhere-bulging-lemma-2-3}
            \item $A^k_m \subset \Delta_{m,k}^{n_k}$ and $B^k_m \subset \Delta_{m,k}^{n_k}$ for all $m \in \set{0,\dots,k}$, \label{prop:nowhere-bulging-lemma-2-4}
            \item the sets $A^k_m$ and $B^k_m$ are disjoint, \label{prop:nowhere-bulging-lemma-2-5}
            \item $\norm{h_k - h_{k-1}}_{H_k} < \frac{1}{2^{k+1}} \min\set{\delta_0, \dots, \delta_k}$, \label{prop:nowhere-bulging-lemma-2-6}
            \item for any entire function $\widetilde{h}$, the inequality $\norm{h_{k-1} - \widetilde{h}}_{H_k} < \delta_k$ implies
            \begin{equation*}
                \pi_z F_{\widetilde{h}}^{n_k}(q_m, w_k) \in A_m^k
            \end{equation*} 
            and 
            \begin{equation*}
                \pi_z F_{\widetilde{h}}^{n_k}(q_m, w'_k) \in B_m^k
            \end{equation*} 
            for all $m \in \set{0, \dots, k}$. \label{prop:nowhere-bulging-lemma-2-7}
        \end{enumerate}
    \end{lemma}

    \begin{proof}[Proof of Lemma~{\ref{lemma:nowhere-bulging-2}}]
        The proof is constructive.
        To start the construction, let $w_0$ and $w'_0$ be any nonzero complex numbers in $D\paren*{0,\frac{1}{4}}$, let $h_0 \coloneqq \id_{\CC}$, and choose any $ 0 < \delta_0  < \frac{1}{4}$ as well as $r_0 > 0$.

        Now, for $k \geq 0$, given any entire functions $h_0, \dots, h_k$, any nonzero complex numbers $w_0, \dots, w_k$ and $w'_0, \dots, w'_k$, and any positive real numbers $r_0, \dots, r_k$ and $\delta_0, \dots, \delta_k$ satisfying properties~\ref{prop:nowhere-bulging-lemma-2-1}--\ref{prop:nowhere-bulging-lemma-2-7}, we construct $h_{k+1}$, $w_{k+1}$, $w'_{k+1}$, $\delta_{k+1}$, and $r_{k+1}$ satisfying the same properties.

        Since $\pi_z F^j_{\widetilde{h}}(z,0) = f^j(z)$ for all $z \in \CC$, $j\in\NN$, and $\widetilde{h}$ entire, by continuity, we can find $w_{k+1}, w'_{k+1} \in D(0,\frac{1}{k+1})\setminus \set{0}$ small enough that 
        \begin{equation}\label{eq:nowhere-bulging-7}
            \pi_z F^{j}_{h_k}(q_m,w_{k+1}) \in \interior{\Delta}^{j}_{m,k+1} \text{ and } \pi_z F^{j}_{h_k}(q_m,w'_{k+1}) \in \interior{\Delta}^{j}_{m,k+1}
        \end{equation}
        for all $m \in \set{0, \dots, k+1}$ and $j \in \set{0, \dots, n_{k+1}}$, and such that
        \begin{equation*}
            \abs{g^{j}(w_{k+1})} < \frac{1}{4} \text{ and } \abs{g^{j}(w'_{k+1})} < \frac{1}{4} 
        \end{equation*} 
        for all $j \in \set{0, \dots, n_{k + 1} + 1}$.
        In particular, from \eqref{eq:nowhere-bulging-7}, we deduce that
        \begin{equation}\label{eq:nowhere-bulging-8}
            \pi_z F^{j}_{h_k}(q_m,w_{k+1}) \in \interior{H}_{k+1} \text{ and } \pi_z F^{j}_{h_k}(q_m,w'_{k+1}) \in \interior{H}_{k+1}
        \end{equation}
        for all $m \in \set{0, \dots, k+1}$ and $j \in \set{0, \dots, n_{k+1} - 1}$. 
        Since $g$ is nonconstant, the zero at $0$ of each $g^j$, $j \in \NN$, is isolated. 
        Hence, up to choosing smaller $w_{k+1}$ and $w'_{k+1}$, we obtain $g^{j}(w_{k+1}) \neq 0$ and $ g^{j}(w'_{k+1}) \neq 0$ for all $j \in \set{0, \dots, n_{k+1} + 1}$.
        Up to slightly perturbing $h_k$, which preserves all conditions obtained so far since these are open, we may assume that the maps
        \begin{equation*}
            w \mapsto \pi_z F^{n_{k+1}}_{h_k}(q_m,w)
        \end{equation*}
        are nonconstant for all $m \in \set{0, \dots, k+1}$.
        Up to slightly changing $w'_{k+1}$, we can then assume that
        \begin{equation*}
            \pi_z F^{n_{k+1}}_{h_k}(q_m,w_{k+1}) \neq \pi_z F^{n_{k+1}}_{h_k}(q_m,w'_{k+1})  
        \end{equation*}
        for all $m \in \set{0, \dots, k+1}$.
        Hence, we can find $r_{k+1} > 0$ such that properties~\ref{prop:nowhere-bulging-lemma-2-4} and \ref{prop:nowhere-bulging-lemma-2-5} hold for $k+1$.

        Then, by continuity of $h \mapsto \pi_z F_h^j(q_m,w)$ at $h_k$ with respect to uniform convergence on $H_{k+1}$ for $w \in \set{w_{k+1}, w'_{k+1}}$ and $m \in \set{0, \dots, k+1}$, which holds by \eqref{eq:nowhere-bulging-8}, there exists $0 < \delta_{k+1} < \frac{1}{4}$ such that $\norm{\widetilde{h} - h_k}_{H_{k+1}} < \delta_{k+1}$ ensures
        \begin{equation*}
            \pi_z F_{\widetilde{h}}^{n_{k+1}}(q_m, w_{k+1}) \in A_m^{k+1}
        \end{equation*} 
        and 
        \begin{equation*}
            \pi_z F_{\widetilde{h}}^{n_{k+1}}(q_m, w'_{k+1}) \in B_m^{k+1}
        \end{equation*} 
        for all $m \in \set{0, \dots, k+1}$.
        This yields property~\ref{prop:nowhere-bulging-lemma-2-7}.

        Lastly, we construct $h_{k+1}$. Observe that $K \coloneqq H_{k+1} \cup \bigcup_{m = 0}^{k+1} \paren*{A_m^{k+1} \cup B_m^{k+1}}$ satisfies all conditions of Runge's Theorem~\ref{thm:runge-1}. 
        Hence, we can approximate the holomorphic function $\Phi \colon K \to \CC$,
        \begin{equation*}
            \Phi(z) \coloneqq
            \begin{cases}
                h_k(z) & \text{if } z \in H_{k+1}, \\
                -\frac{1}{g^{n_{k + 1}}(w_{k+1})} f(\zeta_{k+1, m}) & \text{if } z \in A_m^{k+1} \text{ for some } m, \\
                -\frac{1}{g^{n_{k + 1}}(w'_{k+1})} (f(\zeta'_{k+1, m}) - (k+1)) & \text{if } z \in B_m^{k+1}\text{ for some } m,
            \end{cases}
        \end{equation*}
        up to the error $\frac{1}{2^{k+2}} \min\set{\delta_0, \dots, \delta_{k+1}}$ by an entire function $h_{k+1}$. 
        Altogether, we obtain $w_{k+1}$, $w'_{k+1}$, $\delta_{k+1}$, $r_{k+1}$, and $h_{k+1}$ satisfying properties~\ref{prop:nowhere-bulging-lemma-2-1}--\ref{prop:nowhere-bulging-lemma-2-7}.
    \end{proof}

    \textit{Proof of Lemma~{\ref{lemma:nowhere-bulging-1}} (continued).}
    We use the sequences constructed in \Cref{lemma:nowhere-bulging-2} and show that the sequence $(h_k)_{k\in\NN}$ has a uniform limit on $\CC$. 
    First, we write $h_k$ as a telescoping sum. 
    Fix any $\ell \in \NN$. 
    Then, for every $k > \ell$, we have
    \begin{equation*}
        h_k = h_\ell +  \sum_{j = \ell}^{k-1} \paren{h_{j + 1} - h_{j}}.
    \end{equation*}
    Property~\ref{prop:nowhere-bulging-lemma-2-6} yields
    \begin{equation*}
        \sum_{j = \ell}^{\infty} \norm{h_{j + 1} - h_{j}}_{H_{\ell}} \leq \sum_{j = \ell}^{\infty}  \norm{h_{j + 1} - h_{j}}_{H_{j+1}}  < \sum_{j = \ell}^{\infty} \frac{1}{2^{j+2}} < \infty.
    \end{equation*}
    Therefore, the partial sums form a uniform Cauchy sequence on $H_{\ell}$; hence, they have a uniform limit function $\lim_{k \to \infty} h_k = h_{\ell} + \sum_{j = \ell}^{\infty} (h_{j + 1} - h_{j})$ on $H_{\ell}$. 

    Since the compact sets $H_k$ are an increasing sequence exhausting $\CC$, we can apply Weierstraß's Theorem~\ref{thm:weierstrass} and obtain an entire limit function $h = \lim_{k\to \infty}h_k$. 
    Moreover, for any $k \in \NN$, this limit function satisfies
    \begin{equation}\label{eq:nowhere-bulging-9}
        \begin{split}
            \norm{h - h_k}_{H_{k+1}} &\leq \sum_{j = k}^{\infty} \norm{h_{j + 1} - h_{j}}_{H_{j+1}} 
            \leq \sum_{j = k+1}^{\infty} \frac{1}{2^{j+1}} \min\lbrace \delta_0, \dots, \delta_j \rbrace \\
            &\leq \sum_{j = k+1}^{\infty} \frac{1}{2^{j+1}} \delta_{k+1} < \delta_{k+1}.
        \end{split}
    \end{equation}

    Our construction of $h$ ensures that the map $F_h(z,w) = (f(z) + wh(z), g(w))$, together with the sequences $(w_k)_{k\in\NN}$ and $(w_k')_{k\in\NN}$, satisfies the assertions of \Cref{lemma:nowhere-bulging-1}.

    Indeed, by property~\ref{prop:nowhere-bulging-lemma-2-1} of \Cref{lemma:nowhere-bulging-2}, both $(w_k)_{k\in\NN}$ and $(w_k')_{k\in\NN}$ are sequences of nonzero complex numbers with $\lim_{k\to\infty} w_k = \lim_{k\to\infty} w_k' = 0$. 

    Also, fixing any $k \geq 1$ and $m \in \set{0, \dots, k}$, inequality~\eqref{eq:nowhere-bulging-9} allows us to apply property~\ref{prop:nowhere-bulging-lemma-2-7}, which yields
    \begin{equation*}
        \pi_z F_{h}^{n_k}(q_m, w_k) \in A_m^k \text{ and }  \pi_z F_{h}^{n_k}(q_m, w'_k) \in B_m^k.
    \end{equation*}
    Thus, we obtain
    \begin{equation*}
        \begin{split}
            \abs{\pi_z F_{h}^{n_k+1}(q_m, w_k)} &= \abs{f(\pi_z F_{h}^{n_k}(q_m, w_k)) + g^{n_k}(w_k) h(\pi_z F_{h}^{n_k}(q_m, w_k))}\\
            &\leq \abs{f(\pi_z F_{h}^{n_k}(q_m, w_k)) - f(\zeta_{k,m})}  \\
            & \qquad + \abs{f(\zeta_{k,m}) +  g^{n_k}(w_k) h(\pi_z F_{h}^{n_k}(q_m, w_k))}\\
            &< \frac{1}{4} +  \abs{g^{n_k}(w_k)} \abs*{ h(\pi_z F_{h}^{n_k}(q_m, w_k)) + \frac{1}{g^{n_k}(w_k)} f(\zeta_{k,m})} \\
            &\leq \frac{1}{4} +  \norm*{h + \frac{1}{g^{n_k}(w_k)} f(\zeta_{k,m})}_{A_m^k}
        \end{split}
    \end{equation*}
    where we used \eqref{eq:nowhere-bulging-5} and $\abs{g^{n_k}(w_k)} < 1$.
    Observe that $A_m^k \subset \Delta_{m,k}^{n_k}\subset H_{k+1}$ by property~\ref{prop:nowhere-bulging-lemma-2-4} and \Cref{claim:nowhere-bulging}.
    Applying properties~\ref{prop:nowhere-bulging-lemma-2-1} and \ref{prop:nowhere-bulging-lemma-2-2} together with \eqref{eq:nowhere-bulging-9} then yields
    \begin{equation*}
        \begin{split}
            \norm*{h + \frac{1}{g^{n_k}(w_k)} f(\zeta_{k,m})}_{A_m^k} &\leq  \norm*{h - h_k}_{A_m^k} + \norm*{h_k + \frac{1}{g^{n_k}(w_k)} f(\zeta_{k,m})}_{A_m^k} \\
            &\leq \norm*{h - h_k}_{H_{k+1}} + \frac{1}{4} \\
            &< \delta_{k+1} + \frac{1}{4} < \frac{1}{2}.
        \end{split}
    \end{equation*}
    Therefore, we obtain $\abs{\pi_z F_{h}^{n_{k}+1}(q_m, w_k)} < \frac{3}{4}$, and consequently,
    \begin{equation*}
        \norm{F_{h}^{n_{k}+1}(q_m, w_k)} \leq \abs{\pi_z F_{h}^{n_{k}+1}(q_m, w_k)} + \abs{g^{n_k + 1}(w_k)} < \frac{3}{4} + \frac{1}{4} = 1,
    \end{equation*}
    which proves \eqref{eq:nowhere-bulging-1}.
    For $w'_k$, we observe that, as before,
    \begin{equation*}
        \abs{\pi_z F_{h}^{n_k+1}(q_m, w'_k) - (k+1)} \leq \frac{1}{4} + \norm*{h + \frac{1}{g^{n_k}(w'_k)} (f(\zeta'_{k,m}) - (k+1))}_{B_m^k}
    \end{equation*}
    and 
    \begin{equation*}
        \norm*{h + \frac{1}{g^{n_k}(w'_k)} (f(\zeta'_{k,m}) - (k+1))}_{B_m^k} < \frac{1}{2},
    \end{equation*}
    that is, $\pi_z F_{h}^{n_k+1}(q_m, w'_k) \in D\paren*{k+1, \frac{3}{4}}$.
    Thus,
    \begin{equation*}
        \norm{F_{h}^{n_k+1}(q_m, w'_k)} \geq \abs{\pi_z F_{h}^{n_k+1}(q_m, w'_k)} > k,
    \end{equation*}
    which completes the proof.
\end{proof}

\section{Partial Bulging}\label{sec:partial-bulging}
Transcendental skew products can have Fatou components which are partially bulging in the sense of \Cref{def:bulging}.
We give examples of a partially bulging wandering domain and of a partially bulging Baker domain.
We start with the partially bulging wandering domain.

\begin{theorem}\label{thm:partial-wandering}
    Let $f(z) \coloneqq z + 2\pi i + \frac{1}{2} - \frac{1}{2}\exp(z)$ and let $g$ be a nonconstant entire function with an attracting fixed point at $0$.
    Then there exists an entire function $h$ such that the skew product $F(z,w) = (f(z) + w h(z), g(w))$ has a partially bulging wandering domain.
\end{theorem}

The function $f$ is a modification of Herman's example $z \mapsto z + 2\pi i + 1 - e^z$ (see \cite{Baker1984}), in which the superattracting fixed point of $z \mapsto z + 1 - e^z$ at $0$ is replaced by an attracting one.
The following lemmata summarize the key properties needed below.

\begin{lemma}[{\cite[Section 5]{Baker1984}}]\label{lemma:baker}
    Let $b(z) = z + \frac{1}{2} - \frac{1}{2}\exp(z)$ and $f(z) = b(z) + 2 \pi i$.
    Then the immediate basin of attraction $U$ of $b$ at $0$ is a wandering domain for $f$.
\end{lemma}

The immediate basin $U$ of $b$ contains invariant subsets which are central to our proof.
They are described in the next lemma.

\begin{lemma}\label{lemma:partial-wandering-1}
    Let $b(z) = z + \frac{1}{2} - \frac{1}{2}\exp(z)$.
    For all sufficiently large $r > 0$, the set $D_r \coloneqq D\paren*{ - \frac{1}{2r}, \frac{1}{2r} }$ is a $b$-invariant subset of $U$, that is, $b(D_r) \subset D_r$.
\end{lemma}

\begin{proof}[Proof of Lemma~{\ref{lemma:partial-wandering-1}}]
    We rewrite $b$ as $b(z) = \frac{1}{2} \paren*{ z - \sum_{k = 2}^{\infty} \frac{z^k}{k!}}$.
    Using the conjugation $\zeta = \phi(z) \coloneqq \frac{1}{z}$ in a small punctured neighborhood of $0$, we obtain
    \begin{equation*}
        \begin{split}
            \zeta_1 &= \frac{2}{z \paren*{1 - \sum_{k = 2}^{\infty} \frac{z^{k-1}}{k!}}} = \frac{2}{z} \paren*{1  + \sum_{j=1}^{\infty}\paren*{\sum_{k = 2}^{\infty} \frac{z^{k-1}}{k!}}^j} \\
            &= 2\zeta \paren*{ 1 + \frac{1}{2\zeta} + \mathcal{O}\paren*{\frac{1}{\zeta^2}} } = 2\zeta + 1 + \mathcal{O}\paren*{\frac{1}{\zeta}}.
        \end{split}
    \end{equation*}
    Denote the conjugated function by $\widetilde{b} \coloneqq \phi \circ b \circ \phi$, and observe that left half-planes $L_r \coloneqq \set{ \zeta \in \CC \mid \re (\zeta) < -r }$ are $\widetilde{b}$-invariant for $r > 0$ sufficiently large.
    Indeed,
    \begin{equation*}
        \begin{split}
            \re (\zeta_1) &= 2 \re (\zeta) + 1 + \mathcal{O}\paren*{\frac{1}{\abs{\zeta}}} < -2r + 1 + \frac{r}{2} \leq -r
        \end{split}
    \end{equation*}
    for all $\zeta \in L_r$.
    Moreover, these left half-planes are mapped to disks under $\phi$. 
    More precisely, $\phi(L_r) = D_r = D\paren*{ -\frac{1}{2r}, \frac{1}{2r} }$.
    For $0 < r < R$, it holds that $D_R \subset D_r$.
    Let $r > 0$ be large enough that $L_r$ is $\widetilde{b}$-invariant and $D_r \subset U$. Then $D_r \subset U$ is invariant under $b$ since for every $z \in D_r$,
    \begin{equation*}
        b(z) = \phi \circ \widetilde{b} \circ \phi(z) \in \phi \circ \widetilde{b} \circ \phi(D_r) = \phi \circ \widetilde{b}(L_r) \subset \phi(L_r) = D_r,
    \end{equation*}
    which completes the proof.
\end{proof}

We now prove \Cref{thm:partial-wandering}.

\begin{proof}[Proof of Theorem~{\ref{thm:partial-wandering}}]
    Let $0 < \delta_g < 1$ be such that $g(D(0,\delta_g))\subset D(0,\delta_g)$ and points in $D(0,\delta_g)$ converge uniformly to $0$ under iteration of $g$. 

    By \Cref{lemma:partial-wandering-1}, there exists $r > 0$ large enough that $D_r = D\paren*{ -\frac{1}{2r}, \frac{1}{2r} } \subset U$ is a $b$-invariant disk.
    Fix $0 < \e < \frac{1}{2r}$ and define the disk $K_0 \coloneqq \overline{D\paren*{-\frac{1}{2r}, \frac{1}{2r} - \e}} \subset D_r$.
    Then $K_0 $ is compact and simply connected with nonempty interior. 
    Since $b(K_0) \subset D_r$ is compact, it has positive distance to the boundary $\partial D_r$, and hence, $\e_0 \coloneqq \frac{1}{4} \dist(\partial D_r, b(K_0)) > 0$. 
    Thus, 
    \begin{equation*}
        b(K_0) + \overline{D(0, \e_0)}\subset \overline{D\paren*{-\frac{1}{2r}, \frac{1}{2r} - \e_0}} =: K_1 \subset D_r.
    \end{equation*}
    Similarly, we obtain positive real numbers $\e_n$ such that the compact sets $K_{n+1} \coloneqq \overline{D\paren*{-\frac{1}{2r}, \frac{1}{2r} - \e_{n}}}$ satisfy
    \begin{equation*}
        b(K_n) + \overline{D(0, \e_{n})} \subset K_{n+1} \subset D_r.
    \end{equation*}
    By construction, all sets $K_n$ are compact, simply connected, and contained in $D_r$.
    In particular, there are $p_n > 0$ such that $\min{\abs{K_n}} > p_n$.
    Further shrinking $p_n$ if necessary, we may assume that 
    \begin{equation}\label{eq:uniform-continuity}
        \abs{f(z) - 2(n+1)\pi i} <\frac{1}{2}
    \end{equation}
    for all $z \in D(2n\pi i, p_n)$ and $n \in \NN$.

    We now use an approach similar to \cite[Theorem 1.1]{Potthink2026} (see also the proof of \Cref{thm:nowhere-bulging}) to construct the desired function $h$.
    The idea is to force $(0,0)$ to lie in the Julia set of $F$ while keeping the perturbation small on the sets $K_n$.

    Given a sequence $(w_n)_{n\in\NN}$ of nonzero complex numbers in $D(0,\delta_g)$ with $g^k(w_k) \neq 0$ for all $k \in \NN$, recursively define sets $\Omega_{n,0} \coloneqq \set{0}$ and $\Omega_{n,k+1} \coloneqq f(\Omega_{n,k}) - \frac{g^k(w_n)}{g^k(w_k)} 2(k+1)\pi i + g^k(w_n) D\paren*{0, \frac{1}{2}}$ for $n \in \NN$ and $0 \leq k \leq n-1$.
    By continuity and since $f^k(0) = 2k\pi i$, there exists such a sequence $(w_n)_{n\in\NN}$ of nonzero complex numbers which converges to $0$ and satisfies $\Omega_{n,k} \subset D\paren*{2k\pi i, \frac{p_k}{2}}$ for all $n \in \NN$ and $0 \leq k \leq n$.

    Then, Runge's Theorem~\ref{thm:runge-1} applied to the sets $K_n + 2n\pi i$ and $\overline{ D\paren*{2n\pi i, \frac{p_n}{2}}}$ as shown in Figure~\ref{fig:partial-wandering-sets}, $n \in \NN$, yields an entire function $h$ with
    \begin{equation}\label{eq:partial-wandering-1}
        \abs{h(z)} < \e_n
    \end{equation}
    for all $z \in K_n + 2n\pi i$, $n\in\NN$, and
    \begin{equation}\label{eq:partial-wandering-2}
        \abs*{h(z) + \frac{2(n+1)\pi i}{g^n(w_n)}} < \frac{1}{2}
    \end{equation}
    for all $z \in D\paren*{2n\pi i, \frac{p_n}{2}}$, $n\in\NN$.

    \begin{figure}[ht]
        \centering
        \begin{tikzpicture}[scale=1.1]
            \draw[->] (0,-1.4) -- (0,7.4) node[above] {$\im z$};

            \foreach \n/\e in {0/0.35, 1/0.25, 2/0.18}{
                \pgfmathsetmacro{\y}{3*\n}
                \pgfmathsetmacro{\q}{0.4*\e}
                \draw[dashed] (-1,\y) circle (1);
                \filldraw[fill=gray!30] (-1,\y) circle ({1-\e});
                \filldraw[fill=gray!65] (0,\y) circle (\q);
                \fill (0,\y) circle (0.8pt);
            }

            \node at (-1,0) {$K_0$};
            \node at (-1,3) {$K_1 + 2\pi i$};
            \node at (-1,6) {$K_2 + 4\pi i$};
            \node[right] at (0.2,0)  {$0$};
            \node[right] at (0.15,3) {$2\pi i$};
            \node[right] at (0.12,6) {$4\pi i$};
            \node at (-1,7.5) {$\vdots$};

            \foreach \n in {0,1}{
                \pgfmathsetmacro{\y}{3*\n}
                \draw[->, >=stealth] (-2.1,\y+0.35) to[bend left=40]
                    node[left] {$f$} (-2.1,\y+2.65);
            }
        \end{tikzpicture}
        \caption{The sets $K_n + 2n\pi i$ inside the disks $D_r + 2n\pi i$ together with the disks $D(2n\pi i, p_n/2)$ around $f^n(0) = 2n\pi i$ for $n = 0,1,2$.}
        \label{fig:partial-wandering-sets}
        \end{figure}
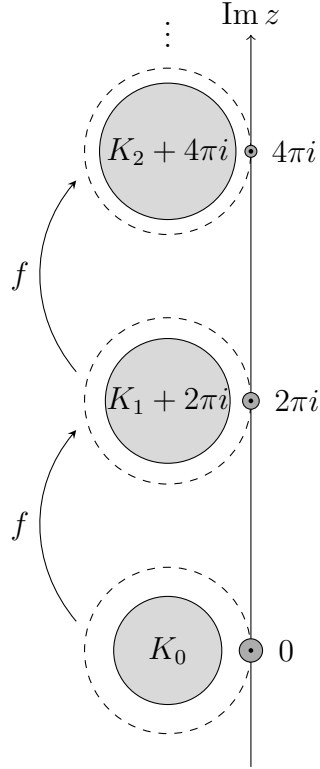

    We first show that $(0,0)$ lies in the Julia set of the skew product $F(z,w) = (f(z) + wh(z), g(w))$.
    Let $z \in \Omega_{n,k}$ for some $k < n$.
    Then, from $\Omega_{n,k} \subset D\paren*{2k\pi i, \frac{p_k}{2}}$ and \eqref{eq:partial-wandering-2}, it follows that
    \begin{equation*}
        \begin{split}
            \pi_z F(z, g^k(w_n)) &= f(z) + g^k(w_n)h(z) \\
            &\in f(\Omega_{n,k}) + g^k(w_n) \paren*{- \frac{2(k+1)\pi i}{g^k(w_k)} + D\paren*{0,\frac{1}{2}}} \\
            &= f(\Omega_{n,k}) - \frac{g^k(w_n)}{g^k(w_k)} 2(k+1)\pi i + g^k(w_n) D\paren*{0, \frac{1}{2}} = \Omega_{n,k+1}.
        \end{split}
    \end{equation*}
    By induction on $k$, it follows that $\pi_z F^k(0, w_n) \in \Omega_{n,k}$ for $0 \leq k \leq n$, using $F^k(0,w_n) = (\pi_z F^k(0,w_n), g^k(w_n))$ and the inclusion just computed.
    Hence, setting $\zeta_n \coloneqq \pi_z F^n(0,w_n) $ so that $ \zeta_n \in D\paren*{2n\pi i, \frac{p_n}{2}}$, this implies
    \begin{equation*}
        \begin{split}
            \abs{\pi_z F^{n+1}(0,w_n)} &= \abs{f(\zeta_n) + g^n(w_n)h(\zeta_n)} \\
            &\leq \abs*{f(\zeta_{n}) - 2(n+1)\pi i} + \abs{g^n(w_n)}\abs*{h(\zeta_n) + \frac{2(n+1)\pi i}{g^n(w_n)}} \\
            &\leq \abs*{f(\zeta_{n}) - 2(n+1)\pi i} + \frac{\abs{g^n(w_n)}}{2} \\  
            &< \frac{1}{2} + \frac{1}{2} = 1,
        \end{split}
    \end{equation*}
    where we used \eqref{eq:uniform-continuity} in the last step.
    Moreover, in the $w$-coordinate, we have $g^{n+1}(w_n) \in D(0,\delta_g)$ for all $n\in\NN$, and consequently, $\norm{F^{n+1}(0,w_n)}$ is uniformly bounded over $n \in \NN$.
    
    Assume that $(0,0)$ lies in the Fatou set of $F$.
    The iterates diverge at $(0,0)$, so any limit along a subsequence is $\infty$. 
    Normality then yields locally uniform divergence on a neighborhood of $(0,0)$, which contradicts $\norm{F^{n+1}(0,w_n)}$ being bounded while $(0,w_n) \to (0,0)$ as $n \to\infty$.
    Therefore, $(0,0)$ lies in the Julia set of $F$.

    Next, we prove that $\interior{K_0} \times \set{0} \subset \F(F)$.
    We use induction to show that 
    \begin{equation}\label{eq:partial-wandering-induction}
        \pi_z F^n(z,w) \in K_n + 2n\pi i
    \end{equation}
    for all $n\in\NN$ and $(z,w) \in K_0 \times D(0,\delta_g)$. 
    
    The base case, $n = 0$, is trivial.
    Assume \eqref{eq:partial-wandering-induction} holds for some $n \in \NN$. 
    Fix $(z,w) \in K_0 \times D(0,\delta_g)$.
    It is clear that $g^n(w) \in D(0,\delta_g)$. 
    By the induction hypothesis, it holds that $\pi_z F^n(z,w) - 2n\pi i \in K_n$, and therefore,
    \begin{equation*}
        \begin{split}
            f(\pi_z F^n(z,w)) &= f(\pi_z F^n(z,w) - 2 n\pi i) + 2 n\pi i \\
            &= b(\pi_z F^n(z,w) -  2 n\pi i) + (n + 1) 2 \pi i \in b(K_n) +  2 (n+1) \pi i.
        \end{split}
    \end{equation*}
    By \eqref{eq:partial-wandering-1}, we obtain
    \begin{equation*}
        \abs{\omega h(\zeta)} = \abs{\omega}\abs{h(\zeta)} \leq \abs{\omega} \e_{n} \leq \e_{n}
    \end{equation*}
    for all $(\zeta,\omega) \in ( K_n + 2n\pi i) \times D(0,\delta_g)$. Consequently,
    \begin{equation*}
        \begin{split}
            \pi_z F^{n+1}(z,w) &= f(\pi_z F^n(z,w)) + g^n(w) h(\pi_z F^n(z,w)) \\
            &\in b(K_n) +  2 (n+1)\pi i + \overline{D(0,\e_{n})} \\
            &\subset K_{n+1} +  2 (n+1)\pi i.
        \end{split}
    \end{equation*}
    This completes the induction.  

    Since $\min \abs{K_n + 2n \pi i} \to \infty$ as $n \to \infty$, we conclude from \eqref{eq:partial-wandering-induction} that the orbits of any points in $K_0 \times D(0,\delta_g)$ diverge uniformly to $\infty$ in the $z$-coordinate and converge uniformly to $0$ in the $w$-coordinate.
    Thus, it follows that $\interior{K_0} \times D(0,\delta_g) \subset \F(F)$.
    Since $(0,0) \in \J(F)$ while $\interior{K_0} \times \set{0} \subset \F(F)$, we have shown that the wandering domain $U$ is partially bulging.
\end{proof}

Baker domains can also be partially bulging, as can be seen from the following example involving Fatou's function.
Recall that Fatou's function is given by $f(z) = z + 1 + e^{-z}$ and has a Baker domain containing the right half-plane $\set{z \in \CC \mid \re z > 0}$.
It was first described by Fatou in \cite{Fatou1926}.

\begin{theorem}\label{thm:partial-baker}
    Let $f(z) \coloneqq z + 1 + e^{-z}$ and let $g$ be a nonconstant entire function with an attracting fixed point at $0$.
    Then there exists an entire function $h$ such that the skew product $F(z,w) = (f(z) + w h(z), g(w))$ has a partially bulging Baker domain. 
    More precisely, there exist two sequences $(S_k)_{k\in\NN}$ and $(T_k)_{k\in\NN}$ of infinite vertical strips, i.e., sets of the form $\set{ z \in \CC \mid a \leq \re(z) \leq b}$ for $a < b$, which are contained in the Baker domain and satisfy
    \begin{itemize}
        \item $f(T_k) \subset T_{k+1}$ and $f(S_k) \subset S_{k+1}$ for all $k \in \NN$,
        \item $T_0 \times \set{0} \subset \J(F)$, and
        \item $\bigcup_{k\in\NN} \interior{S_k} \times \set{0} \subset \F(F)$.
    \end{itemize}
\end{theorem}

\begin{proof}
    First, given two infinite vertical strips $S$ and $T$, we define
    \begin{equation*}
        [S,T] \coloneqq \set{z \in \CC \mid \min_{s\in S}\re(s) \leq \re(z) \leq \max_{t\in T}\re(t)}.
    \end{equation*}
    Furthermore, we write $S < T$ if $\max_{s\in S}\re(s) < \min_{t\in T}\re(t)$.

    It is well known from one-dimensional dynamics that $f$ has a Baker domain $B$ which has $\set{z \in \CC\mid \re(z) \geq R}$ as an absorbing domain (see \Cref{def:absorbing-domain}) for every $R > 0$.
    The right half-plane contains $f$-invariant sequences of infinite vertical strips as detailed in the following claim and shown in Figure~\ref{fig:partial-baker-strips}.
    \begin{claim}\label{claim:partial-baker}
        There exist three sequences $(S_k)_{k\in\NN}$, $(T_k)_{k\in\NN}$, and $(I_k)_{k\in\NN}$ of infinite vertical strips in $\set{z \in \CC \mid \re z > 0}$, a sequence of positive real numbers $(\e_k)_{k\in\NN}$, and $\delta > 0$, such that for all $k \in \NN$,
        \begin{enumerate}[label=(\alph*)]
            \item $S_k < I_k < S_{k+1}$,\label{prop:partial-baker-2}
            \item $T_k + D(0,\delta) \subset I_k$,\label{prop:partial-baker-3}
            \item $f(I_k) \subset I_{k+1}$,\label{prop:partial-baker-4}
            \item $f(T_k) \subset T_{k+1}$, and \label{prop:partial-baker-5}
            \item $f(S_k) + D(0,\e_k) \subset S_{k+1}$. \label{prop:partial-baker-6}
        \end{enumerate}
    \end{claim}

    \begin{figure}[ht]
        \centering
        \begin{tikzpicture}[x=2.6cm, y=1cm]
            \def\a{0.6}   
            \def\H{2}     

            \draw[->] (0,-\H-0.3) -- (0,\H+0.3) node[above] {$\im z$};

            \foreach \k in {0,1,2}{
                \pgfmathsetmacro{\ws}{0.07+0.025*\k}  
                \pgfmathsetmacro{\wt}{0.06+0.02*\k}   
                \pgfmathsetmacro{\wi}{\wt+0.07}       
                \pgfmathsetmacro{\cs}{\a+\k}          
                \pgfmathsetmacro{\ct}{\a+\k+0.5}      

                \fill[gray!10] (\ct-\wi,-\H) rectangle (\ct+\wi,\H);
                \draw (\ct-\wi,-\H) -- (\ct-\wi,\H);
                \draw (\ct+\wi,-\H) -- (\ct+\wi,\H);

                \fill[gray!45] (\ct-\wt,-\H) rectangle (\ct+\wt,\H);

                \fill[gray!45] (\cs-\ws,-\H) rectangle (\cs+\ws,\H);
                \draw (\cs-\ws,-\H) -- (\cs-\ws,\H);
                \draw (\cs+\ws,-\H) -- (\cs+\ws,\H);

                \node[below] at (\cs,-\H) {$S_{\k}$};
                \node[below] at (\ct,-\H) {$T_{\k}$};
                \node[above] at (\ct,\H)  {$I_{\k}$};
            }

            \foreach \k in {0,1}{
                \pgfmathsetmacro{\cs}{\a+\k}
                \pgfmathsetmacro{\ct}{\a+\k+0.5}
                \draw[->, >=stealth] (\cs+0.1,\H+0.3) to[bend left=30]
                    node[above] {$f$} (\cs+0.9,\H+0.3);
                \draw[->, >=stealth] (\ct+0.1,-\H-0.3) to[bend right=30]
                    node[below] {$f$} (\ct+0.9,-\H-0.3);
            }

            \draw[->] (-0.2,0) -- (\a+2.9,0) node[right] {$\re z$};
        \end{tikzpicture}
        \caption{The strips from Claim~\ref{claim:partial-baker} for $k = 0,1,2$, cut off at the top and bottom.
        The function $f$ maps $S_k$ into $S_{k+1}$ and $T_k$ into $T_{k+1}$.}
        \label{fig:partial-baker-strips}
        \end{figure}
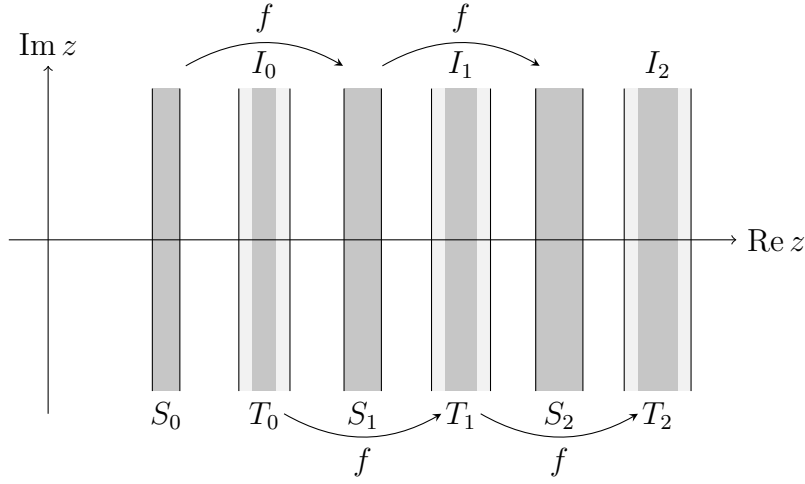
    \begin{subproof}
        Fix $\delta \coloneqq \frac{1}{16}$.
        Let $a > 0$ be such that $\sum_{j=0}^{\infty}e^{-\frac{j}{2} - a} < \delta$ and define $d_k \coloneqq \sum_{j=0}^{k}e^{-\frac{j}{2} - a}$.
        Let $(\e_k)_{k\in\NN}$ be any sequence of positive real numbers with $\sum_{j=0}^{\infty} \e_j < \delta$ and define $e_k \coloneqq \sum_{j=0}^{k-1 } \e_j$.

        Define the vertical strips
        \begin{equation*}
            S_k \coloneqq \set*{ z \in \CC \,\middle|\, a + k - d_k - e_k \leq \re(z) \leq a + k + d_k + e_k },
        \end{equation*}
        \begin{equation*}
            I_k \coloneqq \set*{ z \in \CC \,\middle|\, a + k + \frac{1}{2} - d_k - \delta \leq \re(z) \leq a + k + \frac{1}{2} + d_k + \delta },
        \end{equation*}
        and
        \begin{equation*}
            T_k \coloneqq \set*{ z \in \CC \,\middle|\, a + k + \frac{1}{2} - d_k \leq \re(z) \leq a + k + \frac{1}{2} + d_k }.
        \end{equation*}
        By definition, these strips satisfy properties~\ref{prop:partial-baker-2} and \ref{prop:partial-baker-3}.
        Next, for $k\in\NN$ sufficiently large and $z \in S_k$,
        \begin{equation*}
            \begin{split}
                \re(f(z)) &= \re(z) + 1 + \re(e^{-z}) \\
                &\leq a + k + d_k + e_k + 1 + e^{-k - a + d_k + e_k} < a + (k+1) + d_{k+1} + e_k,
            \end{split}
        \end{equation*}
        and similarly, $\re(f(z)) > a + (k+1) - d_{k+1} - e_k$.
        Thus, property~\ref{prop:partial-baker-6} follows for $k$ large enough.
        Properties~\ref{prop:partial-baker-4} and \ref{prop:partial-baker-5} can be deduced analogously. 
        Up to reindexing, the claim follows.
    \end{subproof}
    \textit{Proof of Theorem~{\ref{thm:partial-baker}} (continued).}
    Observe that all the strips defined in \Cref{claim:partial-baker} lie in $\set{z \in \CC\mid \re(z) \geq R}$ for some $R > 0$, and hence, in an absorbing domain of $B$.
    Let $Q = \set{q_j\mid j\in\NN}$ be a countable dense subset of $T_0$ with pairwise disjoint orbits.
    No point in $Q$ is preperiodic for $f$ since orbits escape to infinity.
    The goal is to apply a construction similar to the proof of \Cref{thm:nowhere-bulging} for this set $Q$ while maintaining good control on the strips $S_k$.

    To that end, we change a few steps of the construction as follows.
    Similarly to \Cref{claim:nowhere-bulging}, but with $n_k = k$, there exists a sequence $(\rho_k)_{k\in\NN}$ of positive real numbers such that letting $H_k \coloneqq [-S_{k-1}, S_{k-1}] \cup I_{k-1}$ for $k\geq 1$, where $-S_{k-1} \coloneqq \set{-z \mid z \in S_{k-1}}$, $H_0 = \emptyset$, and $\Delta^j_{m,k} \coloneqq \overline{D(f^j(q_m), \rho_k)}$, the following hold:
    \begin{enumerate}[label=(\arabic*)]
        \item $\Delta^j_{m,k} \subset I_j$ for all $j,m,k \in \NN$; in particular, $\Delta^j_{m,k} \subset H_k$ for all $j \in \set{0, \dots, k-1}$ and $m,k \in \NN$,
        \item for every $k\in\NN$, all the sets $\Delta^j_{m,k}$, $m = 0,\dots,k$, $j = 0, \dots, k$, are pairwise disjoint,
        \item for every $k\in\NN$,
        \begin{equation*}
            H_k \cup \bigcup_{m = 0}^k \Delta^k_{m,k} \subset H_{k+1},
        \end{equation*}
        \item $H_k$ and $ \Delta^{k}_{0,k} \cup \dots \cup \Delta^{k}_{k,k}$ are disjoint, and\label{prop:partial-baker-7}
        \item for all $z, z' \in \Delta^j_{m,k}$, $k\in\NN$, $m = 0,\dots,k$, and $j = 0, \dots, k$,
        \begin{equation*}
            \abs{f(z) - f(z')} < \frac{1}{4}.
        \end{equation*}
    \end{enumerate}
    Observe, as before, that $H_k \subset H_{k+1}$ and $\CC = \bigcup_{k \in \NN} H_k$.
    Furthermore, $H_k$ is disjoint from the sets $ \Delta^{k}_{0,k}, \dots, \Delta^{k}_{k,k}$ by property~\ref{prop:partial-baker-7}.
    In this construction, the set $H_k$ no longer has connected complement in $\CC$ but, as the disjoint union of two infinite vertical strips, it satisfies the conditions of Arakelyan's Theorem~\ref{thm:arakelyan}.

    We then proceed to prove a variant of \Cref{lemma:nowhere-bulging-2} with property~\ref{prop:nowhere-bulging-lemma-2-6} changed to
    \begin{enumerate}[label=(\roman*$'$), start=7]
        \item $\norm{h_k - h_{k-1}}_{H_k} < \frac{1}{2^{k+1}} \min\set{\delta_0, \dots, \delta_k, \e_0, \dots, \e_k}$ \label{prop:partial-baker-8}
    \end{enumerate}
    with the sequence $(\e_k)_{k\in\NN}$ from \Cref{claim:partial-baker}, and additionally requiring that
    \begin{enumerate}[label=(\roman*), start=9]
        \item $\norm*{h_k}_{S_k} < \frac{\e_k}{2}$.\label{prop:partial-baker-9}
    \end{enumerate}
    The original proof of \Cref{lemma:nowhere-bulging-2} stays almost unchanged.
    First, we start the construction with $h_0 \coloneqq c e^z$ instead of $\id_{\CC}$ with a constant $c > 0$ small enough that $\norm*{h_0}_{S_0} < \frac{\e_0}{2}$.
    Next, observe that $K \coloneqq H_{k+1} \cup \bigcup_{m = 0}^{k+1} \paren*{A_m^{k+1} \cup B_m^{k+1}} \cup S_{k+1}$, with $A_m^{k+1} $ and $ B_m^{k+1}$ as in \Cref{lemma:nowhere-bulging-2}, satisfies all conditions of Arakelyan's Theorem~\ref{thm:arakelyan}.
    Then, the above variant of \Cref{lemma:nowhere-bulging-2} is achieved by adapting the recursive construction of $h_{k+1}$ to approximate the function $\Phi \colon K \to \CC$,
    \begin{equation*}
        \Phi(z) \coloneqq
        \begin{cases}
            h_k(z) & \text{if } z \in H_{k+1}, \\
            -\frac{1}{g^{k + 1}(w_{k+1})} f(\zeta_{k+1, m}) & \text{if } z \in A_m^{k+1} \text{ for some } m, \\
            -\frac{1}{g^{k + 1}(w'_{k+1})} (f(\zeta'_{k+1, m}) - (k+1)) & \text{if } z \in B_m^{k+1}\text{ for some } m, \\
            0 & \text{if } z \in S_{k+1},
        \end{cases}
    \end{equation*}
    by an entire function up to an error of $\frac{1}{2^{k+2}}\min\set{\delta_0, \dots, \delta_{k+1}, \e_0, \dots, \e_{k+1}}$.
    The rest of the proof stays unchanged.
    The construction then yields an entire function $h$ such that $T_0 \times \set{0} \subset \J(F)$ where $F(z,w)\coloneqq(f(z) + wh(z), g(w))$.
    Furthermore, analogously to \eqref{eq:nowhere-bulging-9}, since $S_k \subset H_{k+1}$, it follows from property~\ref{prop:partial-baker-8} that
    \begin{equation*}
        \norm{h - h_k}_{S_k} \leq \sum_{j = k+1}^{\infty} 2^{-(j+1)} \min\set{\delta_0, \dots, \delta_j, \e_0, \dots, \e_j} \leq \sum_{j = k+1}^{\infty} 2^{-(j+1)} \e_k \leq \frac{\e_k}{2},
    \end{equation*}
    and together with \ref{prop:partial-baker-9}, it follows that
    \begin{equation}\label{eq:small-on-s}
        \norm{h}_{S_k} < \e_k
    \end{equation}
    for all $k\in\NN$.
    We show that $\bigcup_{k\in\NN} \interior{S_k} \times \set{0}\subset \F(F)$.
    Since $0$ is an attracting fixed point of $g$, there exist $0 < \alpha < 1$ and $0 < \delta_g < 1$ such that
    \begin{equation*}
        \abs{g(w)} \leq \alpha \abs{w}
    \end{equation*} 
    for all $w \in D(0,\delta_g)$.
    Set $A \coloneqq D(0,\delta_g)$. 
    Then $A$ is a $g$-invariant neighborhood of $0$ such that $g^k(w) \to 0$ as $k \to \infty$ uniformly for all $w \in A$.
    By \eqref{eq:small-on-s} and property~\ref{prop:partial-baker-6} from \Cref{claim:partial-baker}, it follows that for all $k\in\NN$ and for all $(z,w) \in S_k \times A$,
    \begin{equation*}
        \pi_z F(z,w) = f(z) + wh(z) \in f(S_k) + wD(0,\e_k) \subset f(S_k) + D(0,\e_k) \subset S_{k+1}. 
    \end{equation*}
    Since $\min_{s\in S_k}\re(s) \to \infty$ as $k \to \infty$, points in $S_k \times A$ diverge uniformly to $\infty$ in the $z$-coordinate and converge uniformly to $0$ in the $w$-coordinate. 
    We conclude that $\interior{S_k} \times \set{0} \subset \F(F)$.
\end{proof}

\section{Tame Perturbations}\label{sec:tame-perturbation}

In this section, we construct skew products with simply connected escaping wandering domains that are non-bulging despite the fiber dynamics only being perturbed by a comparatively tame perturbation. 
A natural example to keep in mind is a perturbation depending only on $w$.

We show that given a perturbation and some geometric restrictions, there exists a fiber map with an escaping wandering domain that satisfies these geometric restrictions and that is non-bulging, except in the trivial cases $r \equiv 0$ or $g$ constant.

We briefly recall that a \emph{Jordan curve} is a continuous map $\gamma\colon [0,1] \to \CC$ with $\gamma(0) = \gamma(1)$ that is injective on $[0,1)$. 
By the Jordan Curve Theorem, $\CC \setminus \gamma([0,1])$ has exactly two connected components, one bounded and one unbounded.
A \emph{Jordan domain} is a domain in $\CC$ that arises as the bounded component for some Jordan curve $\gamma$.
We can now state the theorem.

\begin{theorem}\label{thm:tame-perturbation}
	Let $(V_n)_{n\in\NN}$ and $(W_n)_{n\in\NN}$ be sequences of Jordan domains in $\CC$ such that 
	\begin{itemize}
		\item the sets $\overline{W_n}$ are pairwise disjoint,
		\item $\overline{V_n} \subset W_n$, and
		\item $\min \abs{\overline{W_n}} \to \infty$ as $n\to\infty$.
	\end{itemize}
	Let $r \colon \CC^2 \to \CC$ be a nonconstant holomorphic function that vanishes on $\set{w=0}$ and let $g$ be a nonconstant entire function with an attracting fixed point at $0$.
	Then there exists an entire function $f$ with a wandering domain $U$ such that the Fatou components $U_n$ containing $f^n(U)$ satisfy $\overline{V_n} \subset U_n \subset W_n$ for all $n\in\NN$ and such that the wandering domain $U$ is non-bulging for the skew product $F(z,w) = (f(z) + r(z,w), g(w))$.
\end{theorem}

Recall the following definition.
\begin{definition}
	We say that a curve $\sigma$ in $\CC$ \emph{surrounds a set $B\subset \CC$} if $B$ is contained in a bounded component of the complement of $\sigma$. For a Jordan curve $\gamma$ we denote by $\operatorname{int} \gamma$ the bounded component of $\CC \setminus \gamma$ and by $\operatorname{ext} \gamma$ the unbounded component.
\end{definition}

We will also need the following notation.

\begin{definition}
    Given $\e > 0$ and curves $\gamma\colon[0,1]\to\CC$ and $\eta\colon[0,1]\to\CC$, we say that $\gamma$ and $\eta$ are \emph{$\e$-close to each other}, written $\gamma \sim_\e \eta$, if their supports are $\e$-close to each other, i.e., $\gamma([0,1]) \subset \eta([0,1]) + D(0,\e)$ and $\eta([0,1]) \subset \gamma([0,1]) + D(0,\e)$.
\end{definition}

The proof of \Cref{thm:tame-perturbation} hinges on the following part of \cite[Theorem~D]{Benini2022}.

\begin{theorem}[{\cite[Theorem~D]{Benini2022}}]\label{thm:theorem-D}
    Let $f$ be a transcendental entire function and suppose that there exist Jordan curves $\gamma_n$ and $\Gamma_n$, $n\ge 0$, a bounded domain $D$, a subsequence $n_k\to\infty$ and  compact sets $L_k$ (associated with $\Gamma_{n_k}$) such that
    \begin{enumerate}[label=(\alph*)]
        \item $\Gamma_n$ surrounds $\gamma_n$, for $n \geq 0$, \label{prop:theorem-D-1}
        \item  for every $k,n, m\geq0$, $m\neq n$ the sets $L_k, \overline{D}, \Gamma_m$  are in $\operatorname{ext}\Gamma_n$, \label{prop:theorem-D-2}
        \item  $\gamma_{n+1}$ surrounds $f(\gamma_n)$, for $n \geq 0$, \label{prop:theorem-D-3}
        \item $f(\Gamma_n)$ surrounds $\Gamma_{n+1}$, for $n \geq 0$, \label{prop:theorem-D-4}
        \item $f(\overline D \cup \bigcup_{k\ge 0} L_k)\subset D$, and \label{prop:theorem-D-5}
        \item$ \max\{\operatorname{dist}(z,L_{k})\mid z \in \Gamma_{n_k}\} = o(\operatorname{dist}(\gamma_{n_k}, \Gamma_{n_k}))\;\text{ as}\; k \to \infty$. \label{prop:theorem-D-6}
    \end{enumerate}
    Then there exists an orbit of simply connected wandering domains $U_n$
    such that $\overline{\operatorname{int} \gamma_n} \subset U_n \subset \operatorname{int}\Gamma_n$, for $n \geq 0$.
\end{theorem}

We also need the following lemma ensuring the existence of outermost points of a Jordan domain in a given direction.

\begin{lemma}\label{lemma:outside-jordan}
	Let $\Omega \subset \CC$ be a Jordan domain and $v \in \CC\setminus\set{0}$. 
    Then there exists $x \in \partial\Omega$ with
	\begin{equation*}
		\overline{\Omega} \cap \mathbb{H}(x, v) = \emptyset
	\end{equation*}
	where
	\begin{equation*}
		\mathbb{H}(x, v) \coloneqq \set{ x + sv + tv^{\perp } \mid s > 0 \text{ and } t \in \RR}
	\end{equation*}
	for any vector $v^{\perp} \in \CC\setminus\set{0}$ perpendicular to $v$.
\end{lemma}

\begin{proof}
	Consider the lines
	\begin{equation*}
		L_s \coloneqq \set{ sv + tv^{\perp}  \mid t \in \RR }
	\end{equation*}
	for $s \in \RR$. Since $\Omega$ is nonempty, the set of intersections
	\begin{equation*}
		I \coloneqq \set{ s \in \RR \mid L_s \cap \overline{\Omega} \neq \emptyset }
	\end{equation*}
	is nonempty. 
    The set $I$, being the projection of the compact set $\overline\Omega$ onto the $v$-direction, is compact.
    Let $s \coloneqq \max I$. 
    Any $x \in L_s \cap \overline{\Omega}$ lies in $\partial \Omega$ and suffices for the statement.
\end{proof}

\begin{proof}[Proof of Theorem~{\ref{thm:tame-perturbation}}]
	This proof is a variant of the construction in \cite[Theorem 5.3]{Benini2022} and, in particular, parts proceed similarly.
    The idea is to construct wandering domains of a prescribed shape that simultaneously contain points near their boundary which, upon adding the given perturbation, are pushed into a region that $f$ maps close to the origin.

	First, let $\delta_g > 0$ be such that $\abs{g(w)} < \abs{w}$ and $g(w) \neq 0$ for all $w \in D(0,\delta_g)\setminus\set{0}$.
    Furthermore, we can write $r(z,w) = w^d h(z,w)$, $d\geq 1$, where $h$ is a holomorphic function that does not vanish identically on $\set{w=0}$.

	The key step of the proof is accomplished in the following lemma.

	\begin{lemma}\label{lemma:tame-perturbation-1}
		There exist an entire function $f$ with a wandering domain $U$, a sequence $(\zeta_n)_{n\in\NN}$ of complex numbers, a sequence $(\rho_n)_{n\in\NN}$ of positive real numbers, a sequence $(w_n)_{n\in\NN}$ of nonzero complex numbers converging to $0$, and a sequence of compact simply connected sets $(K_n)_{n \geq 1}$ such that for all $n \geq 0$,
		\begin{enumerate}[label=(\arabic*)]
			\item $\abs{w_n} < \delta_g$, \label{prop:tame-perturbation-lemma-1-1}
			\item the Fatou component $U_n$ of $f$ containing $f^n(U)$ satisfies $\overline{V_n} \subset U_n \subset W_n$, \label{prop:tame-perturbation-lemma-1-4}
			\item $f(z) + g^{n}(w_n)^d h(z, g^{n}(w_n)) \in K_{n+1}$ for all $ z \in \overline{D(\zeta_{n}, 3\rho_n)} $, \label{prop:tame-perturbation-lemma-1-5}
			\item $f(z) + g^{n}(w_j)^d h(z, g^{n}(w_j)) \in \overline{D(\zeta_{n+1}, 3\rho_{n+1})}$ for all $z \in \overline{D(\zeta_n, 3\rho_n)}$ and $j >  n$, \label{prop:tame-perturbation-lemma-1-6}
			\item $f(z) + g^{n+1}(w_n)^d h(z,g^{n+1}(w_n)) \in \DD $ for all $z \in K_{n+1}$, and \label{prop:tame-perturbation-lemma-1-7}
			\item $\zeta_n \in V_n$. \label{prop:tame-perturbation-lemma-1-extra}
		\end{enumerate}
	\end{lemma}

	\textit{Proof of Theorem~{\ref{thm:tame-perturbation}} (continued).}
    In \Cref{lemma:tame-perturbation-1}, property~\ref{prop:tame-perturbation-lemma-1-4} yields the sequence of wandering domains, whereas properties~\ref{prop:tame-perturbation-lemma-1-5} and \ref{prop:tame-perturbation-lemma-1-6} ensure the existence of points sufficiently close to the boundary to be pushed outside by the perturbation.
    Property~\ref{prop:tame-perturbation-lemma-1-7} forces these points to be mapped close to the origin in the next iteration. 

	Using \Cref{lemma:tame-perturbation-1}, we now proceed to show that the wandering domain $U$ of $f$ is non-bulging for
	\begin{equation*}
		F(z,w) = (f(z) + r(z,w), g(w)) = (f(z) + w^d h(z,w), g(w)).
	\end{equation*}
	We first show that
	\begin{equation}\label{eq:tame-perturbation-1}
		\pi_z F^{n+2}(\zeta_0, w_n) \in \DD
	\end{equation}
	for all $n\in\NN$. 
    Indeed, property~\ref{prop:tame-perturbation-lemma-1-6} applied inductively yields
	\begin{equation*}
		\pi_z F^{n}(\zeta_0, w_n) \in \overline{D(\zeta_n, 3 \rho_n)}.
	\end{equation*}
	Then, with property~\ref{prop:tame-perturbation-lemma-1-5},
	\begin{equation*}
		\begin{split}
			\pi_z F^{n+1}(\zeta_0, w_n) &= f(\pi_z F^{n}(\zeta_0, w_n)) + g^n(w_n)^d h(\pi_z F^{n}(\zeta_0, w_n), g^n(w_n)) \in K_{n+1}.
		\end{split}
	\end{equation*}
	Finally, property~\ref{prop:tame-perturbation-lemma-1-7} implies
	\begin{equation*}
		\pi_z F^{n+2}(\zeta_0, w_n) \in \DD.
	\end{equation*}
	
    Assume that $U$ is bulging.
    Then the iterates of $F$ form a normal family around $(\zeta_0, 0)$ since $\zeta_0 \in V_0 \subset U$ by \ref{prop:tame-perturbation-lemma-1-4} and \ref{prop:tame-perturbation-lemma-1-extra}. 
    In particular, there exist a compact neighborhood $K$ of $(\zeta_0, 0)$ and a subsequence $(F^{n_k + 2})_{k\in\NN}$ of iterates such that this subsequence either diverges uniformly to infinity on $K$ or converges uniformly to a finite limit map. 
	By property~\ref{prop:tame-perturbation-lemma-1-4}, we have $f^n(U) \subset W_n$ and therefore, $\min\abs{W_n} \to \infty$ implies $f^n(z) \to \infty$ as $n\to \infty$ for all $z \in U$.
    Thus, $(F^{n_k + 2})_{k\in\NN}$ has to diverge uniformly to infinity on $K$.
    In particular, there exists $J \in \NN$ with
    \begin{equation*}
        \norm{F^{n_{k} + 2}(z,w)} \geq 2
    \end{equation*}
    for all $k \geq J$ and all $(z,w) \in K$.
	Also, there exists $J'\in\NN$ such that $(\zeta_0, w_{n_k}) \in K$ for all $k \geq J'$ since $\lim_{k\to \infty} w_k = 0$.
	Lastly, let $J'' \in \NN$ be such that for all $k \geq J''$, we have $\abs{g^{n_k + 2}(w_{n_k})} \leq \abs{w_{n_k}} < 1$.
	Then, together with \eqref{eq:tame-perturbation-1}, we obtain for all $k \geq \max\set{J, J', J''}$,
	\begin{equation*}
		2 \leq \norm{F^{n_{k} + 2}(\zeta_0,w_{n_k})} \leq \abs{\pi_z F^{n_{k} + 2}(\zeta_0,w_{n_k})} + \abs{g^{n_k + 2}(w_{n_k})} < 1 + 1 = 2,
	\end{equation*}
	which is a contradiction. Hence, $U$ cannot be bulging.
\end{proof}
We now proceed to show \Cref{lemma:tame-perturbation-1}.
\begin{proof}[Proof of Lemma~{\ref{lemma:tame-perturbation-1}}]
	To construct the function $f$, we need several sequences of constants as provided by the following lemma. 
	\begin{lemma}\label{lemma:tame-perturbation-2}
		There exist sequences $(\zeta_n)_{n\in\NN}$ of complex numbers, $(\rho_n)_{n\in\NN}$ of positive real numbers, and $(w_n)_{n\in\NN}$ of nonzero complex numbers such that for all $n \geq 0$,
		\begin{enumerate}[label=(\roman*)]
			\item $\abs{w_n} < \min\set*{\delta_g, \frac{1}{n+1}}$, \label{prop:tame-perturbation-lemma-2-1}
			\item $\overline{D(\zeta_n, 3\rho_n)} \subset V_n$,\label{prop:tame-perturbation-lemma-2-2}
			\item $\overline{D(\zeta_{n+1}, 3\rho_{n+1})} + g^{n}(w_n)^d h(\overline{D(\zeta_n, 3 \rho_n)}, g^{n}(w_n)) \subset \paren*{W_{n+1} \setminus \overline{V_{n+1}}} \cap \mathbb H_{n+1}$, for an open half-plane $\mathbb H_{n+1}$ with $\mathbb H_{n+1} \cap \overline{V_{n+1}} = \emptyset$,\label{prop:tame-perturbation-lemma-2-3}
			\item $\abs{g^{k}(w_n)}^d \max_{z \in \overline{D(\zeta_k, 3\rho_k)}, w \in \overline{\DD}} \abs{h(z,w)} < \rho_{k+1}$ for all $k \in \set{0,\dots, n-1}$, \label{prop:tame-perturbation-lemma-2-4}
			\item $h(\overline{D(\zeta_n, 3\rho_n)}, \overline{D(0,\abs{g^n(w_n)})}) \subset \mathbb{H}\paren*{\frac{h(\zeta_n, 0)}{2},h(\zeta_n, 0)}$, \label{prop:tame-perturbation-lemma-2-5}  
			\item $h(\zeta_n, 0) \neq 0$, \label{prop:tame-perturbation-lemma-2-6}  
			\item $\abs{g^n(w_n)}^d \sup \abs{h(\overline{D(\zeta_n, 3\rho_n)}, \overline{D(0,\abs{g^n(w_n)})})} < \dist(\partial V_{n+1}, \partial W_{n+1})$, and \label{prop:tame-perturbation-lemma-2-7}  
			\item $ \abs{g^n(w_n)}^d \sup \abs{h(\overline{W_{n+1}}, \overline{\DD})} < \frac{1}{2} $. \label{prop:tame-perturbation-lemma-2-8} 
		\end{enumerate}
	\end{lemma}
	\begin{proof}[Proof of Lemma~{\ref{lemma:tame-perturbation-2}}]
		The sequences are constructed inductively.

		Choose any $\zeta_0 \in V_0$ such that $h(\zeta_0, 0) \neq 0$; hence, property~\ref{prop:tame-perturbation-lemma-2-6} holds. 
        Such an element exists since $z \mapsto h(z, 0)$ is not identically zero. 
        By continuity, there exist $\rho_0 > 0$ and $w_0 \in \CC\setminus\set{0}$ such that $h(\overline{D(\zeta_0, 3\rho_0)}, \overline{D(0,\abs{w_0})}) \subset \mathbb{H}(\frac{h(\zeta_0, 0)}{2},h(\zeta_0, 0))$, i.e., property~\ref{prop:tame-perturbation-lemma-2-5} holds. 
        Further shrinking $\rho_0$ guarantees that $\overline{D(\zeta_0, 3\rho_0)} \subset V_0$, yielding property~\ref{prop:tame-perturbation-lemma-2-2}, and further shrinking $w_0$, we obtain properties~\ref{prop:tame-perturbation-lemma-2-1}, \ref{prop:tame-perturbation-lemma-2-7}, and \ref{prop:tame-perturbation-lemma-2-8}.

        Suppose $\zeta_0, \dots, \zeta_n$, $\rho_0, \dots, \rho_n$, and $w_0, \dots, w_n$ have been constructed satisfying properties~\ref{prop:tame-perturbation-lemma-2-1} through \ref{prop:tame-perturbation-lemma-2-8}.
        We construct the next elements as follows.

		Using property~\ref{prop:tame-perturbation-lemma-2-5}, we first obtain
		\begin{equation*}
			g^{n}(w_n)^d h(\overline{D(\zeta_n, 3 \rho_n)}, g^{n}(w_n)) \subset \mathbb{H}\paren*{g^{n}(w_n)^d\frac{h(\zeta_n, 0)}{2},g^{n}(w_n)^d h(\zeta_n, 0)}.
		\end{equation*}
		Applying \Cref{lemma:outside-jordan} to the Jordan domain $V_{n+1}$ with $v \coloneqq g^{n}(w_n)^d h(\zeta_n, 0) \neq 0$ yields a point $x \in \partial V_{n+1}$ such that 
		\begin{equation*}
			\overline{V_{n+1}} \cap \mathbb{H}\paren*{x, g^{n}(w_n)^d h(\zeta_n, 0)} = \emptyset.
		\end{equation*}
        Set $\mathbb{H}_{n+1} \coloneqq \mathbb{H}\paren*{x, g^{n}(w_n)^d h(\zeta_n, 0)}$.
		Then, for any $z \in V_{n+1} \cap D\paren*{x, \abs*{\frac{g^{n}(w_n)^d h(\zeta_n, 0)}{4}}} \neq \emptyset$, we obtain
		\begin{equation*}
			z + g^{n}(w_n)^d h(\overline{D(\zeta_n, 3 \rho_n)}, g^{n}(w_n)) \subset \paren*{\CC \setminus \overline{V_{n+1}}} \cap \mathbb H_{n+1}.
		\end{equation*}
		Choose any $\zeta_{n+1} \in V_{n+1} \cap D\paren*{x, \abs*{\frac{g^{n}(w_n)^d h(\zeta_n, 0)}{4}}}$ with $h(\zeta_{n+1}, 0) \neq 0$.
		Thus, property~\ref{prop:tame-perturbation-lemma-2-6} holds.
		Furthermore, let $R > 0$ be small enough so that
		\begin{equation*}
			\overline{D(\zeta_{n+1}, 3R)} \subset D\paren*{x, \abs*{\frac{g^{n}(w_n)^d h(\zeta_n, 0)}{4}}}.
		\end{equation*}
		We now show that property~\ref{prop:tame-perturbation-lemma-2-3} holds for all $\rho_{n+1} \leq R$. We have already proven that
		\begin{equation*}
			\overline{D(\zeta_{n+1}, 3\rho_{n+1})} + g^{n}(w_n)^d h(\overline{D(\zeta_n, 3 \rho_n)}, g^{n}(w_n)) \subset  \paren*{\CC \setminus \overline{V_{n+1}}} \cap \mathbb H_{n+1}.
		\end{equation*}
		Using $\overline{D(\zeta_{n+1}, 3\rho_{n+1})} \subset V_{n+1}$ and property~\ref{prop:tame-perturbation-lemma-2-7} shows that
		\begin{equation*}
			\overline{D(\zeta_{n+1}, 3\rho_{n+1})} + g^{n}(w_n)^d h(\overline{D(\zeta_n, 3 \rho_n)}, g^{n}(w_n)) \subset W_{n+1}.
		\end{equation*}
		Taken together, these inclusions yield property~\ref{prop:tame-perturbation-lemma-2-3}.
		Lastly, we construct $\rho_{n+1}$ and $w_{n+1}$. By our choice of $\zeta_{n+1}$, we have $h(\zeta_{n+1},0)\neq 0$. By continuity, there exist $w_{n+1} \in \CC\setminus\set{0}$ and $\rho_{n+1} > 0$ so small that
		\begin{equation*}
			h(\overline{D(\zeta_{n+1}, 3\rho_{n+1})}, \overline{D(0,\abs{g^{n+1}(w_{n+1})})}) \subset \mathbb{H}\paren*{\frac{h(\zeta_{n+1}, 0)}{2},h(\zeta_{n+1}, 0)}.
		\end{equation*}
		Without loss of generality, we may assume that $\rho_{n+1} \leq R$, yielding properties~\ref{prop:tame-perturbation-lemma-2-3} and \ref{prop:tame-perturbation-lemma-2-5}.
		Further shrinking $\rho_{n+1}$ yields property~\ref{prop:tame-perturbation-lemma-2-2}, and further shrinking $w_{n+1}$ yields properties~\ref{prop:tame-perturbation-lemma-2-1}, \ref{prop:tame-perturbation-lemma-2-4}, \ref{prop:tame-perturbation-lemma-2-7}, and \ref{prop:tame-perturbation-lemma-2-8}.
	\end{proof}
	\textit{Proof of Lemma~{\ref{lemma:tame-perturbation-1}} (continued).}
	We use \Cref{lemma:tame-perturbation-2}.
    Set $\delta_0 \coloneqq 1$ and define for $n\in\NN$,
	\begin{equation*}
		\delta_{n+1} \coloneqq \dist(\partial V_{n+1}, \overline{D(\zeta_{n+1}, 3\rho_{n+1})} + g^{n}(w_n)^d h(\overline{D(\zeta_n, 3 \rho_n)}, g^{n}(w_n))).
	\end{equation*}
	It follows from property~\ref{prop:tame-perturbation-lemma-2-3} and
    the compactness of
	\begin{equation*}
		\overline{D(\zeta_{n+1}, 3\rho_{n+1})} + g^{n}(w_n)^d h(\overline{D(\zeta_n, 3 \rho_n)}, g^{n}(w_n))
	\end{equation*}
    that $\delta_{n+1} > 0$.

    We will use \Cref{thm:theorem-D} to obtain wandering domains. 
	Let $\xi \in \CC \setminus \bigcup_{k\in\NN} \overline{W_k}$ and choose $r_{\xi} > 0$ so small that $D \coloneqq D(\xi, r_{\xi}) $ satisfies $\overline{D}\subset \CC \setminus \bigcup_{k\in\NN} \overline{W_k}$.

	Let $\gamma_n \coloneqq \partial V_n$. 
    Choose any Jordan curve $\Gamma_n$ completely inside $W_n \setminus \overline{V_n}$ such that $\gamma_n \sim_{\frac{\delta_n}{2}} \Gamma_n$.
    Moreover, it holds that $\overline{V_n} \subset \operatorname{int}\Gamma_n$.
    Fix any $0 < \alpha_n < \delta_n$ with $\dist(\partial W_n, \Gamma_n) > \alpha_n$ and $\dist(\partial V_n, \Gamma_n) > \alpha_n$.
	Let $W_n'$ be the Jordan domain enclosed by $\Gamma_n$, i.e., let $W_n' \coloneqq \operatorname{int}\Gamma_n$.

	Since $\Gamma_n$ is a Jordan curve, the set $\operatorname{ext}(\Gamma_n) \cap (\Gamma_n + D(0,\e))$ is path connected for every $\e > 0$. 
    Thus, for every $0 < \beta_n < \alpha_n$, we find a Jordan curve $\eta_n$ in $\operatorname{ext}(\Gamma_n) \cap W_n$ such that
	\begin{itemize}
		\item $\Gamma_n \subset \operatorname{int} \eta_n$, and
		\item $\Gamma_n \sim_{\frac{\beta_n}{2}} \eta_n$.
	\end{itemize}
    Since $\eta_n([0,1])$ is compact and lies in the open set $\operatorname{ext}(\Gamma_n) \cap W_n$, the curve $\eta_n$ has positive distance to $\Gamma_n$ and $\partial W_n$.
	We may assume that $\beta_n = o(\alpha_n)$ as $n \to \infty$, e.g., by choosing $\beta_n < \frac{\alpha_n}{n}$.

	Parametrizing $\eta_n$ as a curve defined on $[0,1]$, we find $0 < \ell_n < 1$ such that $\Gamma_n \sim_{\beta_n} \eta_n|_{[\ell_n, 1]}$.
    Let $L_n \coloneqq \eta_n([\ell_n, 1])$.
    Then,
    \begin{equation*}
        \max\set{\dist(z, L_n) \mid z \in \Gamma_n} < \beta_n = o(\alpha_n),
    \end{equation*}
    and observe further that the set $L_n$ is compact and simply connected.

	Consider the Riemann maps $\rho_{W_n'} \colon W_n' \to \DD$ and $\rho_{W_n} \colon W_n \to \DD$ with $\rho_{W_n'}(\zeta_n) = 0$ and $\rho_{W_n}(\zeta_n) = 0$. 
    By Carathéodory's Theorem for the continuation of a Riemann map to the boundary (see \cite[Theorem 6.19]{Zakeri2021}), both maps admit a bijective extension onto the boundary which we denote by the same symbols.

	Since $\rho_{W_n'}(\overline{V_n})$ is a compact subset of the open unit disk, for every $s > 0 $ there exists $d_n = d_n(s) \in \NN$ large enough that $B_n \colon \overline{\DD} \to \overline{\DD}, z \mapsto z^{d_n} $ satisfies 
	\begin{equation*}
		B_n \circ \rho_{W_n'}(\overline{V_n}) \subset \overline{D(0,s)}.
	\end{equation*}
	In particular, since $\rho_{W_{n+1}}(D(\zeta_{n+1},\rho_{n+1}))$ is open and contains $0$, there exists $d_n \in \NN$ such that
	\begin{equation*}
		B_n \circ \rho_{W_n'}(\overline{V_n}) \subset \rho_{W_{n+1}}(D(\zeta_{n+1},\rho_{n+1}))
	\end{equation*}
	or equivalently
	\begin{equation*}
		\rho_{W_{n+1}}^{-1} \circ B_n \circ \rho_{W_n'} (\overline{V_n}) \subset D(\zeta_{n+1},\rho_{n+1}).
	\end{equation*}
	Consider $\Phi_n \colon \overline{W_n'} \to \overline{W_{n+1}}$,
	\begin{equation*}
		\Phi_n(z) \coloneqq \rho_{W_{n+1}}^{-1} \circ B_n \circ \rho_{W_n'}(z).
	\end{equation*}
	Then $\Phi_n$ is continuous on $\overline{W_n'}$ and holomorphic on the interior. Furthermore,
	\begin{itemize}
		\item $\Phi_n(\overline{V_n}) \subset D(\zeta_{n+1},\rho_{n+1}) \subset \overline{D(\zeta_{n+1},2\rho_{n+1})} \subset V_{n+1}$,
		\item $\Phi_n(\Gamma_n) = \partial W_{n+1}$. In particular, $\Phi_n(\Gamma_n)$ surrounds $\Gamma_{n+1}$.
	\end{itemize}
	Choose $0 < \e_n < 1$ such that for every entire function $f$ the inequality $\norm{f - \Phi_n}_{\overline{W'_n}} < \e_n$ implies
	\begin{itemize}
		\item $f(\overline{V_n}) \subset \overline{D(\zeta_{n+1},2\rho_{n+1})}$, and
		\item $f(\Gamma_n) \subset \partial W_{n+1} + D(0,\beta_{n+1})$ and it surrounds $ \Gamma_{n+1} $.
	\end{itemize}
    From $\eta_{n+1} \sim_{\beta_{n+1}/2} \Gamma_{n+1}$ and $\Gamma_{n+1} \sim_{\delta_{n+1}/2} \gamma_{n+1}$, it follows that $\eta_{n+1} \subset \partial V_{n+1} + D(0,\delta_{n+1})$.
    Using property~\ref{prop:tame-perturbation-lemma-2-3} and the definition of $\delta_{n+1}$, we obtain
	\begin{equation*}
		\overline{D(\zeta_{n+1}, 3\rho_{n+1})} + g^{n}(w_n)^d h(\overline{D(\zeta_n, 3 \rho_n)}, g^{n}(w_n)) \subset \operatorname{ext}(\eta_{n+1}).
	\end{equation*}
	Furthermore, since this set is compact, all of its holes are bounded, and since it lies in the half-plane $\mathbb{H}_{n+1}$ outside of $V_{n+1}$, none of its holes can contain $\eta_{n+1}$. 
    Hence, the filled hull $\widetilde{K}_{n+1}$ of this set, i.e., its union with all bounded components of its complement, is a compact simply connected set with $\widetilde{K}_{n+1} \subset W_{n+1} \setminus \overline{\operatorname{int}\eta_{n+1}}$. 
    Since $\widetilde K_{n+1}$ is compact, we can choose $R_{n+1} > 0$ so small that $\widetilde{K}_{n+1} + \overline{D(0,R_{n+1})}$ still lies in the half-plane $\mathbb{H}_{n+1}$ and in $W_{n+1} \setminus \overline{\operatorname{int}\eta_{n+1}}$. 
    By the same argument as above, its filled hull $K_{n+1}$ then also satisfies $K_{n+1} \subset W_{n+1} \setminus \overline{\operatorname{int}\eta_{n+1}}$.
	Let
	\begin{equation*}
		E \coloneqq \overline{D} \cup \bigcup_{n\in\NN} L_n \cup \bigcup_{n\in\NN} \overline{W'_n} \cup \bigcup_{n\in\NN} K_n.
	\end{equation*}
	This set satisfies all conditions of Nersesjan's Theorem~\ref{thm:nersesjan}.
	Define $\Psi \colon E \to \CC$,
	\begin{equation*}
		\Psi(z) \coloneqq \begin{cases}
			\xi & \text{if } z \in \overline{D} \cup \bigcup_{n\in\NN} L_n, \\
			\Phi_n(z) & \text{if } z \in \overline{W'_n} \text{ for some }n, \\
			0 & \text{if } z \in K_n \text{ for some }n,
		\end{cases}
	\end{equation*}
	which is continuous on $E$ and holomorphic on the interior of $E$, and the piecewise constant function $\e \colon E \to (0,\infty)$,
	\begin{equation*}
		\e(z) \coloneqq \begin{cases}
			\min\set*{\frac{r_{\xi}}{2}, \frac{1}{2}} & \text{if } z \in \overline{D} \cup \bigcup_{n\in\NN} L_n \cup \bigcup_{n\in\NN} K_n, \\
			\min\set{\e_n, R_{n+1}, \rho_{n+1}} & \text{if } z \in \overline{W'_n} \text{ for some }n.\\
		\end{cases}
	\end{equation*}
	By \Cref{thm:nersesjan}, there exists an entire function $f$ such that
	\begin{equation*}
		\abs{f(z) - \Psi(z)} < \e(z)
	\end{equation*}
	for all $z \in E$.
	
	This function $f$ satisfies conditions~\textit{\ref{prop:theorem-D-1}}--\textit{\ref{prop:theorem-D-6}} of \Cref{thm:theorem-D}.
	Indeed, conditions~\textit{\ref{prop:theorem-D-1}} and \textit{\ref{prop:theorem-D-2}} follow directly from our choices for the pertinent sets. 
    Condition~\textit{\ref{prop:theorem-D-3}} follows from our choice of $\e_n$ since we have
	\begin{equation*}
		f(\gamma_n) \subset f(\overline{V_n}) \subset \overline{D(\zeta_{n+1}, 2\rho_{n+1})} \subsetneq V_{n+1} = \operatorname{int}\gamma_{n+1}.
	\end{equation*}
	Similarly, \textit{\ref{prop:theorem-D-4}}, and even more directly \textit{\ref{prop:theorem-D-5}}, follow by our choice of the error function. 
    Lastly, \textit{\ref{prop:theorem-D-6}} follows from our choice of $\beta_n$. 
    The function $f$ has to be transcendental since no nonconstant polynomial can satisfy \textit{\ref{prop:theorem-D-5}}, and $f$ is nonconstant by \textit{\ref{prop:theorem-D-3}}.
	
	Consequently, \Cref{thm:theorem-D} yields that $f$ has simply connected wandering domains $U_n$ such that 
	\begin{equation*}
		\overline{V_n} = \overline{\operatorname{int}\gamma_n} \subset U_n \subset \operatorname{int}\Gamma_n = W_n' \subset W_n.
	\end{equation*}
	Hence, property~\ref{prop:tame-perturbation-lemma-1-4} holds.
	Property~\ref{prop:tame-perturbation-lemma-1-1} follows directly from property~\ref{prop:tame-perturbation-lemma-2-1}.
	Furthermore, note that by construction, for all $z \in \overline{D(\zeta_n, 3 \rho_n)}$,
	\begin{equation*}
		\Phi_n(z) + g^n(w_n)^d h(z, g^n(w_n)) \in \widetilde{K}_{n+1},
	\end{equation*}
	and 
	\begin{equation*}
		\abs{\Phi_n(z) - f(z)} < R_{n+1}.
	\end{equation*} 
	Together with the definition of $K_{n+1}$, property~\ref{prop:tame-perturbation-lemma-1-5} follows.

    From \ref{prop:tame-perturbation-lemma-2-4}, it follows that, for $n \in \NN$, $j > n$, and $z \in \overline{D(\zeta_n, 3\rho_n)}$,
    \begin{equation*}
        \begin{split}
            f(z) + g^{n}(w_j)^d h(z, g^{n}(w_j)) &\in f(\overline{V_n}) + g^{n}(w_j)^d h(z, g^{n}(w_j)) \\
            &\subset \overline{D(\zeta_{n+1}, 2\rho_{n+1})} + D(0,\rho_{n+1}) \\
            &\subset \overline{D(\zeta_{n+1}, 3\rho_{n+1})}.
        \end{split}
    \end{equation*}
    Hence, property~\ref{prop:tame-perturbation-lemma-1-6} holds.
    Moreover, observe that by \ref{prop:tame-perturbation-lemma-2-8}, for all $n \in \NN$ and $z \in K_{n+1}$,
    \begin{equation*}
        \begin{split}
            \abs{f(z) + g^{n+1}(w_n)^d h(z,g^{n+1}(w_n))} &\leq\abs{f(z)} + \abs{g^{n}(w_n)}^d\sup\abs{h(\overline{W_{n+1}}, \overline{\DD})} < \frac{1}{2} + \frac{1}{2} = 1,
        \end{split}
    \end{equation*}
    which is property~\ref{prop:tame-perturbation-lemma-1-7}.
    Lastly, property~\ref{prop:tame-perturbation-lemma-1-extra} follows from \ref{prop:tame-perturbation-lemma-2-2}.
    This completes the proof.
\end{proof}

\phantomsection
\label{Bibliography}
\bibliographystyle{amsalpha-doi}
\bibliography{main}

\end{document}